\documentclass[journal]{IEEEtran}

\usepackage{cite}

\ifCLASSINFOpdf
   \usepackage[pdftex]{graphicx}
   \DeclareGraphicsExtensions{.pdf,.jpeg,.png,jpg}
\else
\fi
\usepackage{amsmath}
\usepackage{algorithmic}

\usepackage{array}

\ifCLASSOPTIONcompsoc
\else
  \usepackage[caption=false,font=footnotesize]{subfig}
\fi

\usepackage{stfloats}
\usepackage{url}

\usepackage{amsthm}
\theoremstyle{definition}
\newtheorem{prop}{Proposition}
\newtheorem{assumption}{Assumption}
\newtheorem{lemma}{Lemma}
\newtheorem{definition}{Definition}
\newtheorem{remark}{Remark}
\newtheorem{theorem}{Theorem}
\usepackage{amssymb}
\usepackage{enumitem}

\begin{document}

%
\title{Adaptive Flight Control in the Presence of Limits on Magnitude and Rate}
%
%
%

\author{Joseph~E.~Gaudio,~Anuradha~M.~Annaswamy,~Michael~A.~Bolender, and Eugene~Lavretsky
\thanks{J.E. Gaudio and A.M. Annaswamy are with the Department
of Mechanical Engineering, Massachusetts Institute of Technology, Cambridge,
MA, 02139 USA (email: jegaudio@mit.edu and aanna@mit.edu).}
\thanks{M.A. Bolender is with the Air Force Research Laboratory, WPAFB, OH, 45433 USA.}
\thanks{E. Lavretsky is with The Boeing Company, Huntington Beach, CA, 92647 USA (email: eugene.lavretsky@boeing.com).}
}

\maketitle

\begin{abstract}
Input constraints as well as parametric uncertainties must be accounted for in the design of safe control systems. This paper presents an adaptive controller for multiple-input-multiple-output (MIMO) plants with input magnitude and rate saturation in the presence of parametric uncertainties. A filter is introduced in the control path to accommodate the presence of rate limits. An output feedback adaptive controller is designed to stabilize the closed loop system even in the presence of this filter. The overall control architecture includes adaptive laws that are modified to account for the magnitude and rate limits. Analytical guarantees of bounded solutions and satisfactory tracking are provided. Three flight control simulations with nonlinear models of the aircraft dynamics are provided to demonstrate the efficacy of the proposed adaptive controller for open loop stable and unstable systems in the presence of uncertainties in the dynamics as well as input magnitude and rate saturation.
\end{abstract}


%
\IEEEpeerreviewmaketitle

\section{Introduction}
%
%
%
%
\IEEEPARstart{A}{ny} advanced control system must be capable of incorporating realistic constraints on control inputs such as magnitude limits and rate limits. In flight control, such constraints are commonly present due to actuator limits. While magnitude saturation is often accounted for in the underlying control design, the nonlinearity arising from input rate saturation is often ignored. Actuator rate saturation can lead to Pilot Induced Oscillations (PIO), which expose the aircraft to the risk of failures, and in worst cases, to departing flight \cite{Sofrony_2006,Duda_1995,Amato}. Crashes in the SAAB Grippen development are evidences of the latter \cite{Rundqwist_1996,Rundqwist_1996a}. Additionally, in the event of control surface damage, additional complexities may arise if rate saturation is dominant.

Over the past four decades, adaptive control of linear time-invariant plant models in the presence of parametric uncertainties, perturbations due to bounded disturbances, and unmodeled dynamics has been studied extensively \cite{Narendra_1989,Sastry_1989,Astrom_1995,Ioannou1996,Narendra2005}. Adaptive control in the presence of magnitude constraints has been addressed in \cite{karason1994,Annaswamy_1995,Schwager2005,Jang2009,Lavretsky2007,Serrani2009}. With the first set of results on this topic reported in \cite{karason1994,Annaswamy_1995}, references \cite{Schwager2005,Jang2009} extended the analysis to the multiple-input state feedback case. Design of a state feedback magnitude saturation control architecture with a buffer region was shown in \cite{Lavretsky2007}, by modifying the reference model. Reference \cite{Serrani2009} presents an application dependent architecture. None of these references, however, address input rate saturation.

Rate saturation architectures have been considered in \cite{Hess1997,Barbu_2005,Galeani2008,Kahveci2008,Matsutani2010,Astrom_1989,Astrom_2008,Gaudio_2018}. Reference \cite{Hess1997} proposes directly differentiating the control signal in order to saturate the control rate before re-integrating the signal. References \cite{Barbu_2005,Galeani2008} also propose a non-adaptive, robustness based architectures. These non-adaptive architectures do not directly compensate for plant parametric uncertainty. Reference \cite{Kahveci2008} proposes a state feedback indirect adaptive control architecture, while \cite{Matsutani2010} proposes a direct model reference control architecture, although a matching condition is violated. The integrator anti-windup architecture presented in \cite{Astrom_1989} is proposed for systems with input saturation, but does not include a proof of stability for an adaptive system and indeed may result in instability for certain adaptive systems \cite{Astrom_2008}. The controller presented in this paper is significantly less restrictive than these papers in that it provides a solution for the MIMO case with guarantees of stability. It is assumed that the control input is subjected to a hard limit on its rate and therefore differs from our earlier work in \cite{Gaudio_2018} which imposed only a soft limit, i.e. the control rate was allowed to exceed the specified limits.

The problem addressed in this paper is the control and command tracking of plant dynamics in the presence of parametric uncertainties, using input-output measurements, even when the plant input is subjected to hard limits on its magnitude and rate. In order to introduce a rate limit on the control input without explicit differentiation, a filter is designed with hard saturation nonlinearities, similar to \cite{Sofrony_2006,Amato,Rundqwist_1996,Rundqwist_1996a,Hess1997,Kahveci2008,Matsutani2010}. The introduction of such a filter however causes new challenges in the form of an additional lag. This in turn causes the underlying plant to have an increase in relative degree. This, and the fact that we only have outputs available for feedback, suggests the use of a MIMO adaptive controller that uses output feedback and that is applicable for a plant with higher relative degree. For this purpose, we propose an adaptive control approach that is along the lines of \cite{Zheng2013,Qu2015,Qu2016,Qu2016a}. While the results of these papers guarantee stability, they do not consider either a magnitude or a rate-limit (and will be shown empirically to result in instability in this paper). The contribution of this paper is that a similar controller as in \cite{Zheng2013,Qu2015,Qu2016,Qu2016a}, augmented to account for magnitude and rate limits, can still be shown to lead to bounded solutions even in the presence of these limits.

This paper proceeds as follows: Section \ref{s:Preliminaries} presents mathematical preliminaries. Section \ref{s:Problem_Formulation} describes the problem formulation of the output feedback control a plant with magnitude and rate saturation in the presence of parametric uncertainties. The adaptive control architecture is presented in Section \ref{s:Adaptive_Control_Design}. Stability analysis of the closed-loop system follows in Section \ref{s:Stability_Analysis}. Section \ref{s:Simulations} demonstrates the efficacy of the proposed controller in numerical simulations with fighter and hypersonic aircraft, with conclusions following in Section \ref{s:Conclusion}.

\section{Preliminaries}\label{s:Preliminaries}
The following notation is used for a MIMO plant model with $m$ inputs and $p$ outputs: $\{A,B,C\}:=C(sI-A)^{-1}B$, where $s=\frac{d}{dt}$ is the differential operator. The transmission zeros of this system are defined using Definition 1 of reference \cite{Qu2015}. The columns of the matrix $B$ may be partitioned as $B=[b_1,b_2,\cdots,b_m]$. The input relative degree of the MIMO plant model is stated as follows.
\begin{definition}{\cite{Qu2016}}
	A linear square plant given by $\{A,B,C\}$ has
	\begin{enumerate}[label=\alph*)]
		\item input relative degree $r=[r_1,r_2,...,r_m]^T\in\mathbb{N}^{m\times 1}$ if
		\begin{enumerate}[label=\roman*)]
			\item $\forall j\in\{1,...,m\}$, $\forall k\in\{0,...,r_j-2\}$: $CA^kb_j=0_{p\times 1}$, and
			\item $rank[CA^{r_1-1}b_1\quad CA^{r_2-1}b_2...CA^{r_m-1}b_m]=m$
		\end{enumerate}
		\item uniform input relative degree $r\in\mathbb{N}$ if it has input relative degree $r=[r_1,r_2,...,r_m]^T$ and $r=r_1=r_2=...=r_m$
	\end{enumerate}
\end{definition}
Every MIMO plant model thus has an input relative degree. The following proposition describes a minimal realization where there are differentiators on the plant input.
\begin{prop}{\cite{Qu2016}}
	\label{p:Coord_Change}
	Given a linear system $\{A,B_2,C\}$ with uniform input relative degree 2, the following two realizations are equivalent:
	\begin{enumerate}
		\item $\dot{x}(t)=Ax(t)+B_2(a_1^1s+a_1^0)[u(t)]$, \quad $y(t)=Cx(t)$
		\item $\dot{x}'(t)=Ax'(t)+B_2^1u(t)$, \quad $y(t)=Cx'(t)$
	\end{enumerate}
where 
\begin{equation}
\label{e:x_prime}
x'(t)=x(t)-B_2a_1^1u(t)
\end{equation}
	\begin{equation}
	\label{e:B21}
	B_2^1 = AB_2a_1^1+B_2a_1^0
	\end{equation}
\end{prop}
The following identity holds: $(sI-A)^{-1}s=I+(sI-A)^{-1}A$. Thus for a relative degree two input: $C(sI-A)^{-1}B_2s=C(sI-A)^{-1}AB_2$ as $CB_2$ = 0. In this paper elliptical saturation is considered.
\begin{definition}{\cite{Schwager2005}}
An elliptical saturation function of a vector $v(t)$ is defined as
\begin{equation}\label{e:Es}
E_s(v(t),v_{max}) = 
\begin{cases}
v(t) &\quad||v(t)||<g(v(t))\\
\bar{v}(t) &\quad||v(t)||>g(v(t)) \ 
\end{cases}
\end{equation}
where the function $g(v(t))$ is expressed as
\begin{equation}
g(v(t))=\left(\sum_{i=1}^{m}\left[\frac{\hat{e}_i}{(v_{max})_i}\right]^2\right)^{-1/2}
\end{equation}
where $\hat{e}=\frac{v}{||v||}$ and $\bar{v}=\hat{e}g(v)$ (see Figure \ref{f:Elliptical_Saturation}).
\end{definition}
It can be noted that this saturation function can be alternatively implemented using the projection operator as in \cite{Lavretsky2013}.

\begin{figure}[!t]
	\centering
	\includegraphics[trim={4.2cm 4.2cm 3.9cm 4.2cm},clip,width=0.45\textwidth]{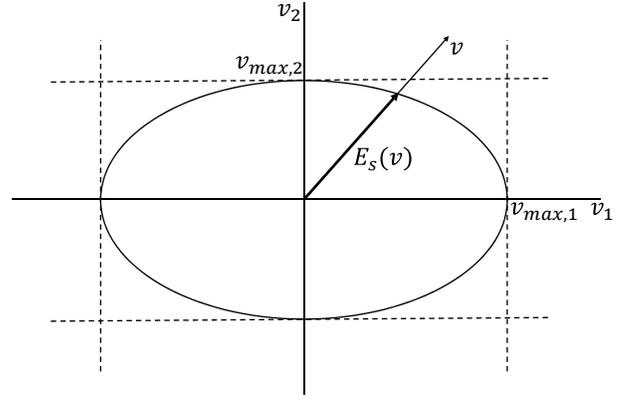}
	\caption{Elliptical saturation function for a two dimensional vector.}
	\label{f:Elliptical_Saturation}
\end{figure}

\section{Problem Formulation}\label{s:Problem_Formulation}
We consider a class of linear plants of the form
\begin{align}
\label{e:plant}
	\begin{split}
		\dot{x}_p(t)&=A_px_p(t)+B_p\Lambda^*[u_p(t)+\Theta_p^{*T}x_p(t)]\\
		y_p(t)&=C_px_p(t)\\
		z(t)&=C_{pz}x_p(t)+D_{pz}\Lambda^*[u_p(t)+\Theta_p^{*T}x_p(t)]
	\end{split}
\end{align}
with known matrices of dimensions: $A_p\in\mathbb{R}^{n_p\times n_p}$, $B_p\in\mathbb{R}^{n_p\times m}$, $C\in\mathbb{R}^{p_p\times n_p}$, $C_{pz}\in\mathbb{R}^{n_z\times n_p}$, and $D_{pz}\in\mathbb{R}^{n_z\times m}$, $n_p\geq p_p\geq m\geq n_z$. The diagonal matrix $\Lambda^*\in\mathbb{R}^{m\times m}$, and $\Theta_p^*\in\mathbb{R}^{n_p\times m}$ represent unknown constant matched uncertainty which enter the plant dynamics through the columns of $B_p$. These uncertainty locations arise in a variety of applications, including in aircraft dynamics \cite{Lavretsky2013}. $x_p$, $u_p$, $y_p$, and $z$ denote the plant state, plant input, plant output, and regulated output respectively. The following assumptions are made of the plant in (\ref{e:plant}):
\begin{assumption}\label{a:minimal}
	$\{A_p,B_p,C_p\}$ is a minimal realization;
\end{assumption}
\begin{assumption}\label{a:tzero}
	All transmission zeros of $\{A_p,B_p,C_p\}$ and $\{A_p,B_p,C_{pz},D_{pz}\}$ are stable;
\end{assumption}
\begin{assumption}\label{a:rel1}
	$\{A_p,B_p,C_p\}$ is uniform relative degree one;
\end{assumption}
\begin{assumption}\label{a:Theta_p_star}
	The uncertainty $\Theta_p^*$ is bounded by a known value, i.e. $||\Theta_p^*||<\Theta_{p,max}$;
\end{assumption}
\begin{assumption}\label{a:Lambda}
	The uncertainty $\Lambda^*$ is diagonal positive definite and bounded by a known value along with its inverse, i.e. $||\Lambda^*||<\Lambda_{max}$, $||\Lambda^{*-1}||<\Lambda_{inv,max}$.
\end{assumption}
Assumptions \ref{a:minimal} and \ref{a:tzero} are standard in output feedback adaptive control \cite{Narendra2005}, where the condition of stable transmission zeros of $\{A_p,B_p,C_{pz},D_{pz}\}$ is employed to ensure controllability of an extended plant in the presence of integral tracking (to be defined) \cite{Lavretsky2013}. The relative degree statement in Assumption \ref{a:rel1} is commonly satisfied for aircraft control systems. The additive uncertainty $\Theta_p^*$ is assumed to be bounded by a known value in Assumption \ref{a:Theta_p_star}. Assumption \ref{a:Lambda} states that the uncertain control effectiveness of each input path are independent of one another, upper bounded by a known value and bounded away from zero by a known value. Known bounds on the norm of the system uncertainty as in Assumptions \ref{a:Theta_p_star} and \ref{a:Lambda} are required given that constraints will be imposed on the plant input.

The goal is to design $u_p$ so that $z$ tracks a bounded command $z_{cmd}$. To ensure a small tracking error, an integral error state $x_e$ is generated as \cite{Lavretsky2013}
\begin{equation}
\label{e:integral_error}
\dot{x}_e(t)=z(t)-z_{cmd}(t).
\end{equation}

\subsection{Rate Saturation}
The goal of this paper is to design an adaptive controller that will accommodate the parametric uncertainties in (\ref{e:plant}) and carry out the desired tracking using a control input which is rate limited, and possibly magnitude limited as well. If the time derivative of the computed control input $u$ were to be available, a filtered control rate $u_r$ could be generated as
\begin{equation}\label{e:u_r_tau}
\tau \dot{u}_r(t) + u_r(t) = \dot{u}(t)
\end{equation}
where $u$ is the output of an adaptive controller. Equation (\ref{e:u_r_tau}) implies that for a small enough filter time constant $\tau$, $u_r(t)\approx \dot{u}(t)$. One could then recover the actual output of the controller $u$ by simply integrating $u_r$ and setting it to be equal to the control input $u_p$ as
\begin{equation}\label{e:u_p_dot_ns}
\dot{u}_p(t)=u_r(t).
\end{equation}
It should be noted that $u_r$ in (\ref{e:u_r_tau}) can be realized without explicit differentiation of $u$, as
\begin{equation}\label{e:filter_laplace}
    u_r(t)=\frac{1}{\tau}u(t)-\frac{1}{\tau}\left[\frac{1}{\tau s+1}\right]u(t)
\end{equation}
where $s$ denotes the operator $d/dt$. Equations (\ref{e:u_r_tau}), (\ref{e:u_p_dot_ns}) provide us with a method for generating a control input $u_p(t)$ that is rate-limited by simply replacing $u_r$ with a saturation function of $u_r$, and generate $u_p(t)$ as
\begin{align}\label{e:u_r_tau_sat}
\tau \dot{u}_r(t) + E_s(u_r(t),u_{r,max}) &= \dot{u}(t)\\
\label{e:dot_u_p}
\dot{u}_p(t)&=E_s(u_r(t),u_{r,max})
\end{align}
instead of equations (\ref{e:u_r_tau}) and (\ref{e:u_p_dot_ns}), where $u_{r,max}$ is the desired rate limit on the control input. As in (\ref{e:filter_laplace}), we note that $u_r$ can be realized without using explicit differentiation as
\begin{equation}\label{e:filtered_temp}
u_r(t)=\frac{1}{\tau}(u(t)-u_p(t))
\end{equation}
which can be used to generate the approximate derivative $u_r$ rather than the differential form (\ref{e:u_r_tau_sat}).

\subsection{Magnitude Saturation}
In addition to rate saturation, it is easy to ensure that the control input $u_p$ is magnitude limited as well. For this purpose, rather than (\ref{e:filtered_temp}), we generate $u_r$ as
\begin{equation}
\label{e:u_r}
u_r(t)=\frac{1}{\tau}(E_s(u(t),u_{max})-u_p(t))
\end{equation}
where $u_{max}$ denotes the magnitude limit of $u$. Equations (\ref{e:dot_u_p}) and (\ref{e:u_r}) lead to a realization of a control input $u_p$ into the plant that is magnitude limited and rate limited, starting from the output $u$ from the adaptive controller (see Figure \ref{f:Block_Diagram} for a schematic).

\begin{figure*}[!t]
	\centering
	\includegraphics[trim={0.2cm 7cm 0.5cm 8.3cm},clip,width=0.9\textwidth]{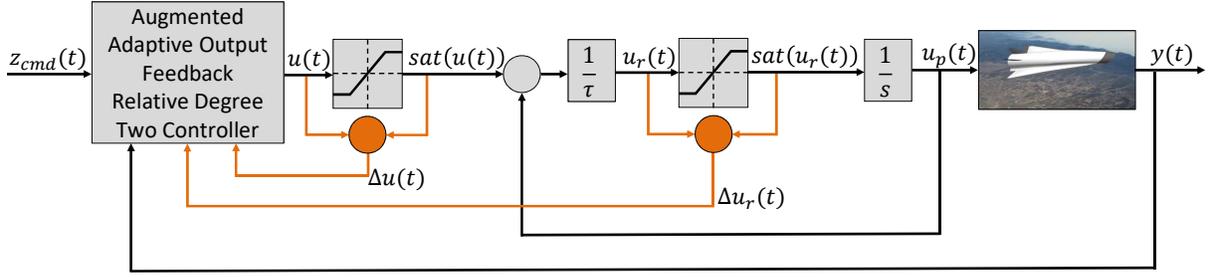}
	\caption{Adaptive controller with input magnitude and rate limiter block diagram.}
	\label{f:Block_Diagram}
\end{figure*}

\subsection{Saturation Effects as Disturbances}
The saturation functions in magnitude and rate can be viewed as two nonlinearities. As in \cite{karason1994,Matsutani2010}, we accommodate these nonlinearities by treating their impact as additive known disturbances. In particular defining two known disturbance terms $\Delta u$ and $\Delta u_r$ as
\begin{align}
\label{e:Control_Deficiency}
\begin{split}
\Delta u(t)&=E_s(u(t),u_{max})-u(t)\\
\Delta u_r(t)&=E_s(u_r(t),u_{r,max})-u_r(t)
\end{split}
\end{align}
it is easy to see that if $u$ does not reach its magnitude saturation limit $u_{max}$, then $\Delta u(t)\equiv 0$. Similarly, if the input rate $u_r$ does not reach its rate saturation limit $u_{r, max}$, then $\Delta u_r(t)\equiv 0$, that is, these known disturbance terms become non-zero only if the magnitude or rate limits are exceeded.

Combining (\ref{e:dot_u_p}), (\ref{e:u_r}), and (\ref{e:Control_Deficiency}) results in the following compact relation between the plant input and the control input:
\begin{equation}
\label{e:filter}
\dot{u}_p(t)=-\frac{1}{\tau}u_p(t)+\frac{1}{\tau}u(t)+\frac{1}{\tau}\Delta u_2(t)
\end{equation}
where $\Delta u_2(t)=(\Delta u(t)+\tau\Delta u_r(t))$ represents the combined effects of magnitude and rate saturation (see Figure \ref{f:Block_Diagram} for a schematic). That is, equation (\ref{e:filter}) determines the relation between $u_p$, the plant input, and $u$, the computed control input in a compact manner.

\subsection{Full Plant Equations}
We now assemble the complete plant model that includes the plant dynamics, the integral of the tracking error, and the effects of rate and magnitude saturation, which are given by equations (\ref{e:plant}), (\ref{e:integral_error}), and (\ref{e:filter}), respectively. This is given by
\begin{equation}\label{e:full_plant_full}
\begin{alignedat}{2}
\underbrace{
\begin{bmatrix}
\dot{x}_p(t)\\
\dot{w}_u(t)\\
\dot{x}_e(t)
\end{bmatrix}
}_{\dot{x}(t)}
&=&&
\underbrace{
\begin{bmatrix}
A_p&B_p&0\\
0&-\frac{1}{\tau}I&0\\
C_{pz}&D_{pz}&0
\end{bmatrix}
}_{A}
\underbrace{
\begin{bmatrix}
x_p(t)\\
w_u(t)\\
x_e(t)
\end{bmatrix}
}_{x(t)}
\\
&&&+
\underbrace{
\begin{bmatrix}
B_p\\
0\\
D_{pz}\\
\end{bmatrix}
}_{B_1}
\Lambda^*\Theta^{*T}_px_p(t)
+
\underbrace{
\begin{bmatrix}
0\\
\frac{1}{\tau}I\\
0
\end{bmatrix}
}_{B_2}
\Lambda^* u(t)
\\
&&&+
\begin{bmatrix}
0\\
\frac{1}{\tau}I\\
0
\end{bmatrix}
\Lambda^*\Delta u_2(t)
+
\underbrace{
\begin{bmatrix}
0\\
0\\
-I\\
\end{bmatrix}
}_{B_z}
z_{cmd}(t)
\\
y(t)&=&~&
\underbrace{
\begin{bmatrix}
C_p&0&0\\
0&0&I
\end{bmatrix}
}_{C}
\begin{bmatrix}
x_p(t)\\
w_u(t)\\
x_e(t)
\end{bmatrix}
\\
z(t)&=&~&
\underbrace{
\begin{bmatrix}
C_{pz}&D_{pz}&0
\end{bmatrix}
}_{C_z}
\begin{bmatrix}
x_p(t)\\
w_u(t)\\
x_e(t)
\end{bmatrix}
+
D_{pz}\Lambda^*\Theta^{*T}_px_p(t).
\end{alignedat}
\end{equation}
where $w_u:=\Lambda^*u_p$. Equation (\ref{e:full_plant_full}) can be expressed in a compact form as
\begin{equation}
\label{e:Plant_Simple}
\begin{alignedat}{2}
\dot{x}(t)&=&~&Ax(t)+B_1\Psi^{*T}_1x(t)+B_2\Lambda^*u(t)\\
&&&+B_2\Lambda^*\Delta u_2(t)+B_zz_{cmd}(t)\\
y(t)&=&~&Cx(t)\\
z(t)&=&~&C_zx(t)+D_{pz}\Psi^{*T}_1x(t)
\end{alignedat}
\end{equation}
where $A\in\mathbb{R}^{n\times n}$, $B_1\in\mathbb{R}^{n\times m}$, $B_2\in\mathbb{R}^{n\times m}$, $B_z\in\mathbb{R}^{n\times n_z}$, $C\in\mathbb{R}^{p\times n}$, $C_{z}\in\mathbb{R}^{n_z\times n}$. The matrix of additive uncertainty is represented as $\Psi^{*T}_1=
\begin{bmatrix}
\Lambda^*\Theta^{*T}_p&0&0
\end{bmatrix}$ and the measured output is $y$.

Based on the structure of the full plant model in (\ref{e:full_plant_full}) it can be noted that $\{A,B_2,C\}$ is a minimal realization given Assumptions \ref{a:minimal} and \ref{a:tzero}. In particular Assumption \ref{a:tzero} ensures that the inclusion of integral tracking (\ref{e:integral_error}) in the extended dynamics (\ref{e:full_plant_full}) preserves controllability given that $\{A_p,B_p,C_{pz},D_{pz}\}$ does not have a transmission zero at the origin \cite{Lavretsky2013}. Observability of the extended dynamics (\ref{e:full_plant_full}) is preserved (and may also be enhanced) with the inclusion of integral tracking (\ref{e:integral_error}) as the integral tracking state $x_e$ is measured. Furthermore it can be demonstrated that the inclusion of the filter dynamics with state $w_u$ in (\ref{e:full_plant_full}) maintains controllability and observability of the extended system, given the form of a low-pass filter in series with a controllable and observable plant with states $x_p$ and $x_e$.

Additionally, given Assumptions \ref{a:tzero} and \ref{a:rel1}, all transmission zeros of $\{A,B_2,C\}$ are stable and $\{A,B_2,C\}$ is uniform relative degree two. It can be noticed that $B_1$ can be spanned by a linear combination of $B_2$ and $AB_2$ as $B_1=\tau AB_2+B_2$. Furthermore $\Psi^*_1$ satisfies $\Psi^{*T}_1B_2=0$ ($\Psi^*_1$ is not in the same input path as the computed control input) and is bounded by a known value from Assumptions \ref{a:Theta_p_star} and \ref{a:Lambda}, i.e. $||\Psi^*_1||<\Psi_{max}=\Theta_{p,max}\Lambda_{max}$.

\section{Adaptive Control Design}\label{s:Adaptive_Control_Design}
This section presents an adaptive controller for the plant model of uniform relative degree two in equation (\ref{e:Plant_Simple}) whose structure is similar to that in \cite{Zheng2013,Qu2015,Qu2016,Qu2016a}. As will be shown, modifications in the underlying adaptive laws are however needed to account for the effects of input magnitude and rate saturation. Section \ref{s:Control_Architecture} presents the control architecture with gains designed in Section \ref{s:LRS}. The construction of an underlying error model that is strictly positive real is presented in Section \ref{s:SPR}.

\subsection{Control Architecture}\label{s:Control_Architecture}
As in any adaptive control design, we begin the control architecture discussion with the introduction of a reference model, which denotes the desired dynamics from the plant when there is no parametric uncertainty and no saturation disturbance. This corresponds to the plant model in (\ref{e:Plant_Simple}) with $\Lambda^*=I$, $\Psi_1^*=0$, and $\Delta u_2(t)=0$. A closed-loop term is added, as in recent investigations \cite{Lavretsky2012,Lavretsky2013,Wise2013,Wise_2018,Lavretsky_2019,Gibson2012,Gibson2013a,Gibson2013,Gibson2014}, which leads to a closed-loop reference model (CRM) in order to ensure smooth control inputs. Given that this problem is in output feedback, the reference model additionally serves the dual purpose as an observer \cite{Lavretsky2012,Lavretsky2013}. The use of $x_m$ to denote the state of the CRM as compared to the common observer notation of $\hat{x}$ is solely a notation choice, where both representations have appeared in the adaptive control literature. Based on the plant structure (\ref{e:Plant_Simple}) in Section \ref{s:Problem_Formulation}, a CRM may be chosen as
\begin{align}
\label{e:Reference_Model}
\begin{split}
\dot{x}_m(t)&=Ax_m(t)\hspace{-.01cm}+\hspace{-.01cm}B_2u_{bl}(t)\hspace{-.01cm}+\hspace{-.01cm}L(y(t)-y_m(t))\hspace{-.01cm}+\hspace{-.01cm}B_zz_{cmd}(t)\\
y_m(t)&=Cx_m(t)\\
z_m(t)&=C_zx_m(t)+D_{pz}u_{bl}(t)
\end{split}
\end{align}
where the design of $L$ is discussed in Section \ref{s:LRS}. The input $u_{bl}$ is chosen so as to ensure a stable CRM, and is of the form
\begin{equation}
\label{e:u_bl}
u_{bl}(t)=-Kx_m(t)
\end{equation}
where the matrix $K\in\mathbb{R}^{n\times m}$ can be designed using linear quadratic regulator (LQR) techniques \cite{Lavretsky2013} to provide for a desired reference model dynamics matrix $A_m=A-B_2K$. It can be noticed that in the presence of perfect output tracking of the reference model output ($y(t)-y_m(t)=0$), the CRM in (\ref{e:Reference_Model}) has the form of the standard adaptive control open loop reference model with $\dot{x}_m(t)=A_mx_m(t)+B_zz_{cmd}(t)$.

Given that the plant in (\ref{e:Plant_Simple}) contains parametric uncertainties, an adaptive control input will be used to counter these uncertainties. The difficulty however is that the introduction of the filter in (\ref{e:filter}) causes the underlying plant dynamics to have a relative degree of two, given that the model in (\ref{e:plant}) has a relative degree of one as noted in Assumption \ref{a:rel1}. This causes the corresponding error model, derived by subtracting (\ref{e:Reference_Model}) from (\ref{e:Plant_Simple}), to have a relative degree two property as well. The relative degree two property prohibits the use of a typical adaptive control input of the form of $u(t)=\Omega^T(t)\xi(x_m(t))$, where $\Omega$ is an estimate of unknown parameters and $\xi$ is the regressor. In order to provide for a strictly positive real (SPR) error model, an extra zero must be added as discussed in \cite{Narendra2005} Chapter 5. For this purpose, we now design the computed control input $u$ as \cite{Qu2016a}
\begin{equation}
\label{e:Control_Input}
u(t)=(a^1_1s+a^0_1)\Omega^T(t)\bar{\xi}(t)
\end{equation}
where the variables $a^1_1>0$ and $a^0_1>0$ can be selected as desired to place the extra zero. The adaptive parameter matrix $\Omega$ and regressor vector $\bar{\xi}$ are defined as
\begin{align}
\label{e:Omega_Xi_bar}
\begin{split}
\Omega(t)&=[\Lambda^T(t), \Psi_1^T(t), \Psi_2^T(t)]^T\\
\bar{\xi}(t)&=[\bar{u}_{bl}^T(t), -x_m^T(t), -\bar{x}_m^T(t)]^T
\end{split}
\end{align}
where $\Omega$ is an estimate of $\Omega^{*}=[\Lambda^{*-1}, \bar{\Psi}^{*T}_1, \bar{\Psi}^{*T}_2]^T$, and $\bar{\Psi}_1^{*}=
\begin{bmatrix}
\frac{\tau}{a_1^1}\Theta_p^{*T}&0&0
\end{bmatrix}^T$ and $\bar{\Psi}_2^{*}=
\begin{bmatrix}
\left(I-\tau\frac{a_1^0}{a_1^1}I\right)\Theta_p^{*T}&0&0
\end{bmatrix}^T$.
The following filtered signals are defined as
\begin{align}
\label{e:Filtered_Signals}
\begin{split}
\bar{u}_{bl}(t)&=\frac{1}{a_1^1s+a_1^0}u_{bl}(t)\\
\Delta\bar{u}_2(t)&=\frac{1}{a_1^1s+a_1^0}\Delta u_2(t)\\
\bar{x}_m(t)&=\frac{1}{a_1^1s+a_1^0}x_m(t)
\end{split}
\end{align}
Since the elements of $\bar{\xi}$ are filtered signals, it is easy to show that (\ref{e:Control_Input}) can be realized without explicit differentiation, provided the derivative of $\Omega$ can be directly synthesized. We will show in the following that this is indeed the case.

\subsection{The Design of $L$}\label{s:LRS}
We now address the design of $L$, which is accomplished so as to ensure an underlying transfer function to be SPR. Since we begin with a non-square plant, a lemma is needed to square-up $\{A,B_2,C\}$ in order to apply the KYP lemma (see \cite{Narendra2005} Lemma 2.5). See \cite{Qu2016a} for a proof and design procedure of the contents of Lemma \ref{l:Squareup}.
\begin{lemma}{\cite{Qu2016a}}
	\label{l:Squareup}
	For plant models of the form of (\ref{e:Plant_Simple}) which satisfy Assumptions \ref{a:minimal} to \ref{a:rel1}, there exists a matrix $B_{s1}\in\mathbb{R}^{n\times(p-m)}$ such that $\{A,\bar{B}_2,C\}$, where $\bar{B}_2=[B_2,B_{s1}]$, has stable transmission zeros and nonuniform input relative degree $r_i=2$ for $i=1,2,...,m$ and $r_i=1$ for $i=m+1,m+2,...,p.$
\end{lemma}
The matrix $B_{s1}$ is used only in the design of $L$ as follows:
\begin{equation}\label{e:B_part}
\bar{B}_2^1=[B_2^1,B_{s1}]
\end{equation}
\begin{equation}
\label{e:S}
S^T=(C\bar{B}_2^1)=[S_2^T,S_1^T]
\end{equation}
\begin{equation}
\bar{C}=SC
\end{equation}
\begin{equation}
R^{-1}(\varepsilon)=(\bar{C}\bar{B}_2^1)^{-1}\left[\bar{C}A\bar{B}_2^1+(\bar{C}A\bar{B}_2^1)^T\right](\bar{C}\bar{B}_2^1)^{-1}+\varepsilon I
\end{equation}
\begin{equation}
\label{e:varepsilon}
\varepsilon>\varepsilon_{max}(A,B,\bar{C},a_1^1,a_1^0,\Lambda_{max},\Psi_{max})
\end{equation}
\begin{equation}
\label{e:L}
L=\bar{B}_2^1R^{-1}(\varepsilon)S
\end{equation}
\begin{equation}
\label{e:A_L_Star}
A^*_L=(A+B_1\Psi^{*T}_1-LC).
\end{equation}
where $\varepsilon>\varepsilon_{max}$ is selected large enough to guarantee that $A_L^*$ is Hurwitz and the realization $\{A^*_L,B,SC\}$ is SPR. In summary, the closed-loop gain $L$ is chosen so as to guarantee that $\{A^*_L,B,SC\}$ is SPR. This property in turn will be used to derive the adaptive laws for adjusting $\Omega$, which is shown in the next section.

\subsection{SPR Error Model}\label{s:SPR}
This section will derive an error model and propose adaptive update laws from its SPR properties. The model tracking error is defined as: $e_x(t)=x(t)-x_m(t)$. Applying (\ref{e:Plant_Simple}), (\ref{e:Control_Input}), (\ref{e:Filtered_Signals}), (\ref{e:Reference_Model}) and (\ref{e:A_L_Star}), the following error model may be derived:
\begin{equation}
\label{e:e_x_Full}
\begin{alignedat}{2}
\dot{e}_x(t)&=&~&A^*_Le_x(t)+B_2^1\Lambda^*\bar{\Psi}^{*T}_1x_m(t)\\
&&&+B_2\Lambda^*(a^1_1s+a^0_1)\bar{\Psi}^{*T}_2\bar{x}_m(t)\\
&&&-B_2\Lambda^*(a^1_1s+a^0_1)\Lambda^{*-1}\bar{u}_{bl}(t)\\
&&&+B_2\Lambda^*(a^1_1s+a^0_1)\Omega^T(t)\bar{\xi}(t)\\
&&&+B_2\Lambda^*(a^1_1s+a^0_1)\Delta \bar{u}_2(t).
\end{alignedat}
\end{equation}
Given the undesirable differentiators in the right hand side of this equation, Proposition \ref{p:Coord_Change} may be applied with the modified model tracking error variable of the form
\begin{equation}
\label{e:e_mx_representation}
\begin{alignedat}{2}
e_{mx}(t)&=&~&e_x(t)-B_2\Lambda^*a_1^1\left[\bar{\Psi}^{*T}_2\bar{x}_m(t)+\Omega^T(t)\bar{\xi}(t)\right.\\
&&&\left.-\Lambda^{*-1}\bar{u}_{bl}(t)+\Delta \bar{u}_2(t)\right].
\end{alignedat}
\end{equation}
This results in a modified model tracking error model as
\begin{align}
\label{e:e_mx}
\begin{split}
\dot{e}_{mx}(t)&=A^*_Le_{mx}(t)+B_2^1\Lambda^*\tilde{\Omega}^T(t)\bar{\xi}(t)+ B_2^1\Lambda^*\Delta \bar{u}_2(t)\\
e_y(t)&=Ce_{mx}(t)=Ce_x(t)
\end{split}
\end{align}
where $\tilde{\Omega}(t)=\Omega(t)-\Omega^*$.

It should be noted that equation (\ref{e:e_mx}) cannot be used directly to determine the rules for adjusting the adaptive parameters $\Omega(t)$. Unlike the approaches in \cite{Zheng2013,Qu2015,Qu2016,Qu2016a}, there are additional terms $\Delta \bar{u}_2(t)$ present here due to the presence of magnitude and rate saturation. This introduces a significant departure from the procedure in \cite{Zheng2013,Qu2015,Qu2016,Qu2016a} and requires additional tools. It should also be noted that the methods proposed in \cite{karason1994,Annaswamy_1995,Schwager2005,Jang2009,Lavretsky2007,Serrani2009} are not sufficient in determining adaptive laws for the stability of the overall system. How these difficulties are overcome in this MIMO setting correspond to the main contributions of this paper and are elaborated on below.

In order to eliminate the effects of $\Delta\bar{u}_2$, an implementable auxiliary signal $e_{\Delta}$ is introduced as
\begin{align}
\label{e:e_auxiliary}
\begin{split}
\dot{e}_{\Delta}(t)&=A_Le_{\Delta}(t)+B_2^1\Omega_{\Delta}^T(t)\bar{\xi}_{\Delta}(t)\\
e_{y,\Delta}(t)&=Ce_{\Delta}(t).
\end{split}
\end{align}
The matrix $\Omega_{\Delta}$ and vector $\bar{\xi}_{\Delta}$ are
\begin{align}
\label{e:Omega_Xi_bar_Delta}
\begin{split}
\Omega_{\Delta}(t)&=[\hat{\Lambda}^T(t), \hat{\Psi}_1^T(t), \hat{\Psi}_2^T(t)]^T\\
\bar{\xi}_{\Delta}(t)&=[\Delta\bar{u}_2^T(t), e_{\Delta}^T(t), \bar{e}_{\Delta}^T(t)]^T
\end{split}
\end{align}
where $\Omega_{\Delta}$ is an estimate of $\Omega_{\Delta}^{*}=[\Lambda^*, \Lambda^*\bar{\Psi}^{*T}_1, \Lambda^*\bar{\Psi}^{*T}_2]^T$, and where the following filtered signal is defined:
\begin{equation}\label{e:e_delta_bar}
\bar{e}_{\Delta}(t)=\frac{1}{a_1^1s+a_1^0}e_{\Delta}(t).
\end{equation}
The following is an equivalent representation of (\ref{e:e_auxiliary}), which is obtained using (\ref{e:A_L_Star}) and (\ref{e:e_delta_bar}):
\begin{equation}
\begin{alignedat}{2}
\dot{e}_{\Delta}(t)&=&~&A^*_Le_{\Delta}(t)-B_2^1\Lambda^*\bar{\Psi}_1^{*T}e_{\Delta}(t)\\
&&&-B_2(a_1^1s+a_1^0)\Lambda^*\bar{\Psi}_2^{*T}\bar{e}_{\Delta}(t)+B_2^1\Omega_{\Delta}^T(t)\bar{\xi}_{\Delta}(t).
\end{alignedat}
\end{equation}
To eliminate the differentiator on the right hand side of this equation, Proposition \ref{p:Coord_Change} can be once again applied with a modified auxiliary signal of the form
\begin{equation}
e_{m\Delta}(t)=e_{\Delta}(t)+B_2a_1^1[\Lambda^*\bar{\Psi}_2^{*T}\bar{e}_{\Delta}(t)].
\end{equation}
This results in modified dynamics of the form
\begin{align}
\begin{split}
\dot{e}_{m\Delta}(t)&=A^*_Le_{m\Delta}(t)+B_2^1\tilde{\Omega}_{\Delta}^T(t)\bar{\xi}_{\Delta}(t)+B_2^1\Lambda^*\Delta\bar{u}_2(t)\\
e_{y,\Delta}(t)&=Ce_{m\Delta}(t)=Ce_{\Delta}(t)
\end{split}
\end{align}
where $\tilde{\Omega}_{\Delta}(t)=\Omega_{\Delta}(t)-\Omega_{\Delta}^*$.
To remove the effects of $\Delta\bar{u}_2$ the following modified augmented error ($e_{mu}(t)=e_{mx}(t)-e_{m\Delta}(t)$) model is introduced:
\begin{align}
\label{e:eu_dot}
\begin{split}
\dot{e}_{mu}(t)&=A^*_Le_{mu}(t)+B_2^{1}\Lambda^*\tilde{\Omega}(t)\bar{\xi}(t)-B_2^1\tilde{\Omega}_{\Delta}^T(t)\bar{\xi}_{\Delta}(t)\\
e_{y,u}(t)&=e_y(t)-e_{y,\Delta}(t)
\end{split}
\end{align}
We note that $e_{y,u}(t)$ is available at each $t$ since $e_y(t)$ is measurable and $e_{y,\Delta}(t)$ is a known signal that can be computed at each time $t$.
\begin{lemma}{\cite{Qu2016a}}
	\label{l:SPR}
	Given $L\in\mathbb{R}^{n\times p}$ in (\ref{e:L}) and $S\in\mathbb{R}^{p\times p}$ in (\ref{e:S}), $\varepsilon$ in (\ref{e:varepsilon}), and with Assumptions \ref{a:minimal} to \ref{a:Lambda}, for plant models of the form of (\ref{e:Plant_Simple}), the transfer function $\{A^*_L,\bar{B}_2^1,SC\}$  is strictly positive real.
\end{lemma}
Lemma \ref{l:SPR} also implies that $\{A^*_L,B_2^1,S_2C\}$ is SPR given the partitions in (\ref{e:B_part}) and (\ref{e:S}). Thus given the structure of the modified augmented error model in (\ref{e:eu_dot}), the adaptive parameters $\Omega$ and $\Omega_{\Delta}$ can be updated as
\begin{equation}
\label{e:Omega_dot}
\dot{\Omega}(t)=-\Gamma_{\Omega}\bar{\xi}(t)e_{y,u}^T(t)S_2^T
\end{equation}
\begin{equation}
\label{e:Omega_Delta_dot}
\dot{\Omega}_{\Delta}(t)=\Gamma_{\Omega_{\Delta}}\bar{\xi}_{\Delta}(t)e_{y,u}^T(t)S_2^T
\end{equation}
where $\Gamma_{\Omega}>0$ and $\Gamma_{\Omega_{\Delta}}>0$ are adaptive update gains.

A few comments regarding the two update laws are in order. Equation (\ref{e:Omega_dot}) represents a standard update to address parametric uncertainties in $\Omega^*$. Equation (\ref{e:Omega_Delta_dot}) represents an update due to the effects of saturation, since once the update law and $e_{\Delta}(0)$ are initialized at zero, they will only change when a saturation limit is reached. The adaptive update laws in (\ref{e:Omega_dot}) and (\ref{e:Omega_Delta_dot}) along with the fact that the realization $\{(A^*-LC,B_2^1,S_2C)\}$ is SPR provides the foundation for the closed-loop system stability. This is established in the next section.

\section{Stability Analysis}\label{s:Stability_Analysis}

Before proceeding to state the main stability result, a discussion of boundedness of the adaptive parameters is presented in Section \ref{s:BAP}. The boundedness of all remaining states is addressed in Section \ref{s:BoS}. It can be noted that due to the presence of saturation, not all signals can be tracked and thus results are local in nature. Bounds on the local stability results will be shown to be proportional to the level of saturation and uncertainty.

\subsection{Boundedness of Adaptive Parameters}\label{s:BAP}

We consider the following Lyapunov function candidate:
\begin{equation}
\label{e:Lyapunov}
\begin{alignedat}{2}
V(e_{mu}(t), \tilde{\Omega}(t), \tilde{\Omega}_{\Delta}(t))&=&~&e_{mu}^T(t)Pe_{mu}(t)\\
&&&\hspace{-.1cm}+Tr\left[\tilde{\Omega}^T(t)\Gamma^{-1}_{\Omega}\tilde{\Omega}(t)|\Lambda^*|\right]\\
&&&\hspace{-.1cm}+Tr\left[\tilde{\Omega}_{\Delta}^T(t)\Gamma^{-1}_{\Omega_{\Delta}}\tilde{\Omega}_{\Delta}(t)\right]
\end{alignedat}
\end{equation}
where $P=P^T>0$ is a positive definite matrix which can be used to guarantee the SPR properties of $\{(A^*_L,\bar{B}_2^1,SC)\}$ and satisfy
\begin{align}
\label{e:SPR}
\begin{split}
A^{*T}_LP+PA^*_L&=-Q<0\\
P\bar{B}_2^1&=C^TS^T
\end{split}
\end{align}
for a positive definite matrix $Q=Q^T>0$. The following partition may be used:
\begin{equation}
\label{e:Partition}
P[B_2^1,B_{s1}]=C^T[S_2^T,S_1^T].
\end{equation}
Taking a derivative of (\ref{e:Lyapunov}) with respect to time:
\begin{equation}
\begin{alignedat}{2}
\dot{V}(e_{mu}(t), \tilde{\Omega}(t), \tilde{\Omega}_{\Delta}(t))&=&~&e_{mu}^T(t)(A^{*T}_LP+PA^*_L)e_{mu}(t)\\
&&&+2e_{mu}^T(t)PB_2^1\Lambda^*\tilde{\Omega}(t)\bar{\xi}(t)\\
&&&-2e_{mu}^T(t)PB_2^1\tilde{\Omega}_{\Delta}(t)\bar{\xi}_{\Delta}(t)\\
&&&+2Tr\left[\tilde{\Omega}^T(t)\Gamma^{-1}_{\Omega}\dot{\Omega}(t)|\Lambda^*|\right]\\
&&&+2Tr\left[\tilde{\Omega}_{\Delta}^T(t)\Gamma^{-1}_{\Omega_{\Delta}}\dot{\Omega}_{\Delta}(t)\right].
\end{alignedat}
\end{equation}
Applying (\ref{e:Omega_dot}), (\ref{e:Omega_Delta_dot}), (\ref{e:SPR}), and (\ref{e:Partition}):
\begin{equation}
\dot{V}(e_{mu}(t), \tilde{\Omega}(t), \tilde{\Omega}_{\Delta}(t))=-e_{mu}^T(t)Qe_{mu}(t)\leq0.
\end{equation}
Thus all of $(e_{mu}, \tilde{\Omega}, \tilde{\Omega}_{\Delta})$ are bounded. This does not show that the model tracking error $e_x$ is bounded, however the following relation for the tracking error can be obtained:
\begin{equation*}
||e_x(t)||=\mathcal{O}[\sup_{\zeta\leq t}||\Delta u_2(\zeta)||].
\end{equation*}
Such a relation is similar to that derived in \cite{karason1994,Annaswamy_1995,Schwager2005,Jang2009,Lavretsky2007}, where the constants in the big $\mathcal{O}$ notation are system parameters.

\subsection{Boundedness of the State}\label{s:BoS}

For clarity of presentation, the constant matrices, constant scalars and time-varying scalars found in this section are defined in Appendix \ref{app:constants}. The following closed-loop system representation can be obtained through combining (\ref{e:u_r}), (\ref{e:Control_Deficiency}), (\ref{e:Reference_Model}), (\ref{e:u_bl}), (\ref{e:Control_Input}), (\ref{e:Omega_Xi_bar}), (\ref{e:Filtered_Signals}), and (\ref{e:e_mx}):
\begin{equation}
\label{e:Chi_full}
\begin{alignedat}{2}
\dot{\chi}(t)&=&~&A_{cl}\chi(t)-B_{\Omega}\tilde{\Omega}^{T}(t)C_{\bar{\xi}}\chi(t)\\
&&&+C_{\Delta \bar{u}_2}\frac{1}{a_1^1}\Delta u_2(t)+B_Zz_{cmd}(t)
\end{alignedat}
\end{equation}
where the full state of the system to be shown bounded is:
\begin{equation}
\label{e:chi}
\chi(t)=
\begin{bmatrix}
\bar{x}_m^T(t)&x_m^T(t)&e_{mx}^T(t)&\Delta \bar{u}_2^T(t)
\end{bmatrix}^T.
\end{equation}
It can be noted from the block upper triangular structure of $A_{cl}$ in Appendix \ref{app:constants}, that its eigenvalues can be placed arbitrarily stable. This is due to the controllability properties of the underlying dynamics and the use of a closed-loop reference model (\ref{e:Reference_Model}). Showing boundedness of the state $\chi$ will result in boundedness of all of the signals in the closed-loop system given that $e_{mu}$, $\tilde{\Omega}$, and $\tilde{\Omega}_{\Delta}$ were shown to be bounded in Section \ref{s:BAP}. The main challenges that arise here are due to the various combinations of scenarios of magnitude and rate saturations that can occur, which causes $\Delta \bar{u}_2$ to be non-zero. This in turn necessitates the use of multiple sub-cases that are discussed in the proof of Theorem \ref{t:One} in Appendix \ref{app:proof}. Before proceeding to a discussion of the main stability theorem, given that the system is input constrained, the following reference command bound is required.
\begin{assumption}\label{a:z}
The reference command $z_{cmd}$ is bounded as $||z_{cmd}(t)||\leq z_{cmd,max}$ and is chosen such that $\rho \chi_{min}<\chi_{max}$, $\kappa_3>0$, $\kappa_8>0$, $\kappa_5^2>4\kappa_4\kappa_6$, $\kappa_5^2>4\kappa_9\kappa_{10}$.
\end{assumption}
Assumption \ref{a:z} implies that the magnitude of the reference command is upper bounded with respect to the level of saturation and system uncertainty. The larger the level of saturation and smaller the system uncertainty, the larger the allowable command. The bound $\rho \chi_{min}<\chi_{max}$ can always be achieved for a provably bounded command, as for example if $z_{cmd,max}=0$: $\chi_{min}=0$ and $\chi_{max}$ is finite, with $\kappa_3>0,\kappa_8>0$ and $\kappa_5^2>4\kappa_4\kappa_6$, $\kappa_5^2>4\kappa_9\kappa_{10}$. Thus $\chi_{min}$ can be seen to represent a shift in the equilibrium of the system due to tracking of $z_{cmd}$. We now state the main theorem of stability.

\begin{theorem}
	\label{t:One}
	Under Assumption \ref{a:z} for the system in (\ref{e:Plant_Simple}) and (\ref{e:Reference_Model}), control input (\ref{e:Control_Input}), known control disturbances (\ref{e:Control_Deficiency}), adaptive laws in (\ref{e:Omega_dot}), (\ref{e:Omega_Delta_dot}), and Lyapunov function (\ref{e:Lyapunov}), the closed-loop system has bounded trajectories for all $t\geq0$ if the following two conditions are satisfied: 
	\begin{enumerate}[label=(\roman*)]
		\item $|\chi(0)|<\frac{1}{\rho}\chi_{max}$
		\item $\sqrt{V(0)}<\sqrt{\frac{\lambda_{min}}{\gamma_{max}}}\bar{\Omega}_{max}$
	\end{enumerate}
	
	Further
	
	$$||\chi(t)||<\chi_{max},\quad\forall t>0$$
	And
	$$||e_x(t)||= \mathcal{O}[\sup_{\zeta\leq t}||\Delta u_2(\zeta)||]$$
\end{theorem}
\begin{proof}
Given that the dynamics matrix $A_{cl}$ can be chosen by design to be strictly stable as in Appendix \ref{app:constants}, it satisfies the following:
$$A_{cl}^TP_{cl}+P_{cl}A_{cl}<-Q_{cl}$$
where the matrix $P_{cl}$ may be computed using a positive definite matrix $Q_{cl}$. Define the following candidate Lyapunov function of the closed-loop dynamics:
\begin{equation}
\label{e:Lyap}
W(\chi(t))=\chi^T(t)P_{cl}\chi(t).
\end{equation}
Define the following level set, $\mathcal{B}$, of $W$:
\begin{equation}\label{e:B}
\mathcal{B}:\left\{\chi(t)|W(\chi(t))=p_{min}\chi_{max}^2\right\}.
\end{equation}
The following annulus region is defined:
\begin{equation}
\mathcal{A}:\left\{\chi(t)|\chi_{min}\leq ||\chi(t)||\leq \chi_{max}\right\}.
\end{equation}
The proof of boundedness of the full system state follows from two steps. The first step will show that $\mathcal{B}\subset\mathcal{A}$ using condition \textit{(ii)} of Theorem 1. Step 2 will show that $\dot{W}(\chi(t))<0, \forall \chi\in \mathcal{A}$. Condition \textit{(i)} in Theorem 1 implies the following:
\begin{equation}\label{e:lesslyap}
W(\chi(0))<W(\mathcal{B}).
\end{equation}
The two steps thus result in the following:
\begin{equation}\label{e:lesslyap2}
W(\chi(t))\leq W(\chi(0))\quad\forall t\geq0.
\end{equation}
Theorem 1 follows from these two steps.

\textit{Proof Step 1:} This step shows that $\mathcal{B}\subset\mathcal{A}$. Using condition \textit{(ii)} from Theorem 1, it can be seen that $\tilde{\Omega}_{max}<\bar{\Omega}_{max}$. Additionally by Assumption \ref{a:z}:
$$\rho \chi_{min}<\chi_{max}.$$
Equation (\ref{e:Lyap}) can be used to show that $W(\chi(t))$ is bounded from below by $p_{min}||\chi(t)||^2\leq W(\chi(t))$. This implies the following:
$$||\chi(t)||\leq \chi_{max}\quad\forall \chi\in \mathcal{B}.$$
In a similar manner from equation (\ref{e:Lyap}), $W(\chi(t))$ can be bounded from above by $W(\chi(t))\leq p_{max}||\chi(t)||^2$. These relations imply the following:
$$\chi_{min}<\frac{1}{\rho}\chi_{max}\leq||\chi(t)||, \quad \forall \chi\in \mathcal{B}.$$
The definition of the annulus region $\mathcal{A}$ can then be used to conclude that $\mathcal{B}\subset\mathcal{A}$.

\textit{Proof Step 2:} It can then be shown that $\dot{W}(\chi(t))< 0$ $\forall \chi\in\mathcal{A}$. Three cases are considered. Case 1 considers the system not in magnitude nor rate saturation ($\Delta u=0,\Delta u_r=0$). Case 2 considers the system in magnitude saturation ($\Delta u\neq0,\Delta u_r=0$) and case 3 considers the system in rate saturation ($\Delta u=0,\Delta u_r\neq0$). The proof of this step is located in Appendix \ref{app:proof}.

Boundedness of the state $\chi$ in equation (\ref{e:chi}) as well as boundedness of $e_{mu}$, $\tilde{\Omega}$, and $\tilde{\Omega}_{\Delta}$ in Section \ref{s:BAP} is sufficient to show that all of the signals in the closed-loop system remain bounded, thus concluding the proof of Theorem \ref{t:One}.
\end{proof}

\begin{remark}
	Global stability cannot be achieved in the presence of magnitude and rate saturation for general (possibly open loop unstable) plant models. There always exists initial conditions that will cause the system to have unbounded trajectories regardless of the design of the controller. The results are thus dependent on formulations of regions of attraction, dependent on the saturation levels and amount of uncertainty, such as those presented in Theorem \ref{t:One}.
\end{remark}
\begin{remark}
	If the plant dynamics are open-loop stable, equation (\ref{e:plant}) is bounded-input, bounded-output (BIBO) stable and equations (\ref{e:u_r}) and (\ref{e:dot_u_p}) results in a bounded control input. Thus the state trajectory $\chi$ is bounded regardless of initial condition.
\end{remark}
\begin{remark}
	This architecture incorporates the effects of saturation in the auxiliary signal in (\ref{e:e_auxiliary}). This auxiliary signal changes in the presence of saturation and thus alters the control input through the adaptive update laws. The effect of saturation may also be incorporated in the reference model as done in the $\mu$-mod architecture \cite{Lavretsky2007}. A similar error model can be formulated for the $\mu$-mod architecture, thus the MIMO output feedback magnitude and rate saturation architecture presented this paper can be extended to the $\mu$-mod architecture.
\end{remark}

\section{Flight Control Numerical Simulations}
\label{s:Simulations}

This section presents numerical simulation results for an open loop stable nonlinear F-16 vehicle model for single-input-single-output (SISO) longitudinal dynamics and multiple-input-multiple-output (MIMO) lateral-directional dynamics, as well a numerical simulation of a nonlinear hypersonic vehicle model for open loop unstable longitudinal dynamics. For the F-16 vehicle, the proposed adaptive controller maintains stability in the presence of input saturation and demonstrates increased performance compared to a non-adaptive controller. The benefit of the proposed adaptive controller is even more apparent for the open loop unstable hypersonic vehicle dynamics, where the proposed adaptive controller provides for command tracking, whereas non-adaptive and standard adaptive responses are unacceptable.

\subsection{F-16 Numerical Simulation}
\label{ss:F16}
This section applies the adaptive controller with limiter presented in this paper to two nonlinear F-16 simulations adapted from \cite{Stevens2003,Russell_2003,Nguyen_1979}. The aircraft simulations feature nonlinear equations of motion with aerodynamic forces and moments calculated using aerodynamic parameters scheduled from aerodynamic look-up tables based on the flight condition. A trim point for this nonlinear vehicle model was obtained at a straight and level flying condition at a velocity of $500$ ft/s with an altitude of $15,000$ ft. The vehicle model was linearized about this trim point in order to obtain linearized dynamics for control design as in (\ref{e:plant}). The actuator time constants and saturation levels from the simulation documentation \cite{Russell_2003} were employed in this simulation, thus providing for physically consistent input constraints. In the following numerical simulations, the thrust input to the aircraft was held constant (at the trim value). It is common in the aerospace industry to separate the design of the full control system of fighter aircraft into longitudinal and lateral-directional dynamics \cite{Lavretsky2013}.

The longitudinal dynamics of an aircraft describe the short timescale motion in the pitch plane. The longitudinal variables for control design are given by
\begin{equation*}
x_p=
\begin{bmatrix}
\alpha&q
\end{bmatrix}^T
\>~
u_p=\delta_e
\>~
y_p=q
\>~
z_p=q
\end{equation*}
where the longitudinal state is composed of the vehicle's angle of attack $\alpha$ and pitch rate $q$. Pitch rate is both the measured and regulated variable. An elevator deflection $\delta_e$ represents the input to the dynamics. The elevator gain parameter $\tau$ and saturation levels are as follows:
\begin{equation*}
\tau=0.0495
\>~
u_{max}=\pm 25~\text{deg}
\>~
u_{r,max}=\pm 60~\text{dps}.
\end{equation*}

The lateral-directional dynamics govern motion in the roll and yaw axes and were simulated to demonstrate the MIMO case. The lateral-directional subsystem is given by
\begin{equation*}
x_p=
\begin{bmatrix}
\beta&p&r
\end{bmatrix}^T
\>~
u_p=
\begin{bmatrix}
\delta_a&\delta_r
\end{bmatrix}^T
\>~
y_p=
\begin{bmatrix}
p&r
\end{bmatrix}^T
\>~
z_p=p
\end{equation*}
where the lateral-directional state is composed of the vehicle's angle of sideslip $\beta$, roll rate $p$, and yaw rate $r$. Roll rate and yaw rate are both measured while roll rate is the regulated variable. Aileron $\delta_a$ and rudder $\delta_r$ deflections represents the inputs to the dynamics. The aileron gain parameter $\tau$ and saturation levels are as follows:
\begin{equation*}
\tau=0.0495
\>~
u_{max}=\pm 21.5~\text{deg}
\>~
u_{r,max}=\pm 80~\text{dps}.
\end{equation*}
The rudder gain parameter $\tau$ and saturation levels are as follows:
\begin{equation*}
\tau=0.0495
\>~
u_{max}=\pm 30~\text{deg}
\>~
u_{r,max}=\pm 120~\text{dps}.
\end{equation*}

Separate controllers were designed for each of the two subsystems, longitudinal and lateral-directional. Integral command tracking (\ref{e:integral_error}) was included for both subsystems in addition to the explicit inclusion of the dynamics of each magnitude and rate saturated actuator as in (\ref{e:filter}), resulting in complete plant models as in (\ref{e:full_plant_full}), (\ref{e:Plant_Simple}). The plant models for each subsystem satisfy Assumptions \ref{a:minimal}-\ref{a:rel1}, as required for the adaptive control design in this paper.

The first numerical simulation in Figures \ref{f:state_2} and \ref{f:control_1} show the longitudinal dynamics of the F-16 aircraft in the presence of a $75\%$ decrease in the total pitching moment coefficient $C_m$. In this simulation, the pitch rate command is stepped every two seconds. The significant amount of uncertainty in $C_m$, results in large degradation of performance for the non-adaptive baseline-only controller (\ref{e:u_bl}), where it can be seen in Figure \ref{f:state_2} that the baseline controller does not settle before the step command changes. This is prohibitive in the design of guidance controllers for aircraft which rely on step responses close to the desired response. This motivates the use of an adaptive control input in order to recover the desired performance, even in the presence of uncertainty.

To this end, the adaptive control input presented in this paper was designed as in (\ref{e:Control_Input}). In order to counter the uncertainty and recover the desired performance, the adaptive controller is seen to require large control rates, which surpass $200$ dps and violate the physical actuator limits as seen in Figure \ref{f:control_1}. If the input limits are enforced for the adaptive control input as in \cite{Zheng2013,Qu2015,Qu2016,Qu2016a}, which does not include the modified error model (\ref{e:eu_dot}) and adaptive update law (\ref{e:Omega_Delta_dot}), the closed loop system can be seen to be unstable. This is the motivation for the explicit accounting for saturation nonlinearities in the design of the adaptive controller presented in this paper. Figure \ref{f:state_2} shows a stable response for the proposed adaptive controller using the modified error model (\ref{e:eu_dot}) and saturation update law (\ref{e:Omega_Delta_dot}), as described in this paper. In particular, the adaptive controller presented in this paper has a significantly improved pitch rate response as compared to the baseline controller. This improved response is due to the faster response of the control input as seen by the input rate profile in Figure \ref{f:control_1}. The adaptive controller proposed in this paper allows for significant performance recovery in the presence of the input limits. Figure \ref{f:long_error} shows the integrated absolute value of the difference in pitch rate responses between the ideal not-limited adaptive control response and each of the baseline response and limited adaptive control response. In this figure the adaptive controller can be seen to provide for approximately half the integrated error. It should be noted that some accumulated error is expected given the constraints on the control input. Stability is maintained in the adaptive controller presented in this paper by explicitly taking into account extra disturbance due to rate saturation as seen in Figure \ref{f:long_error}.

\begin{figure}[!t]
	\centering
	\includegraphics[trim={1.7cm 3.8cm 1.9cm 4.1cm},clip,width=\linewidth]{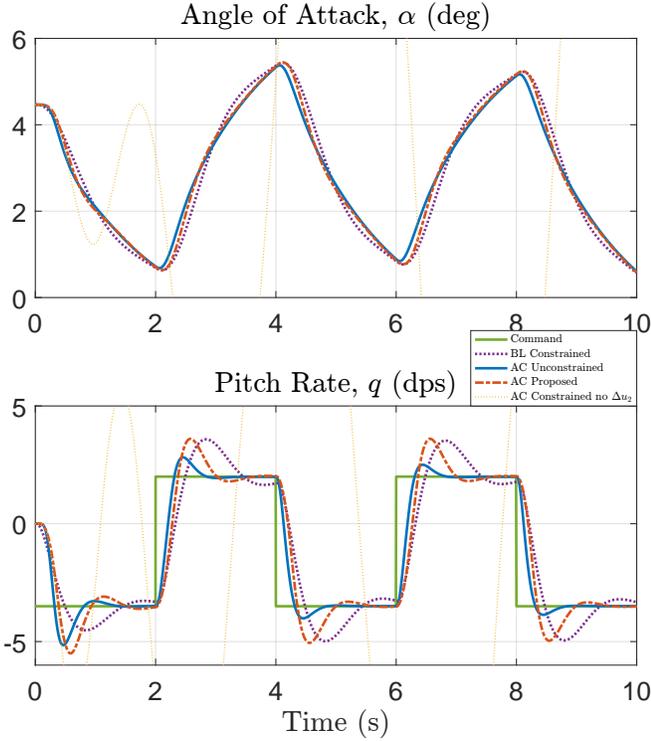}
	\caption{F-16 aircraft. Longitudinal state response comparison between baseline only control (BL Constrained), unconstrained adaptive control (AC Unconstrained), and constrained adaptive control with (AC Proposed) and without (AC Constrained no $\Delta u_2$) update law saturation modification.}
	\label{f:state_2}
\end{figure}
\begin{figure}[!h]
	\centering
	\includegraphics[trim={1.7cm 3.8cm 1.9cm 4.1cm},clip,width=\linewidth]{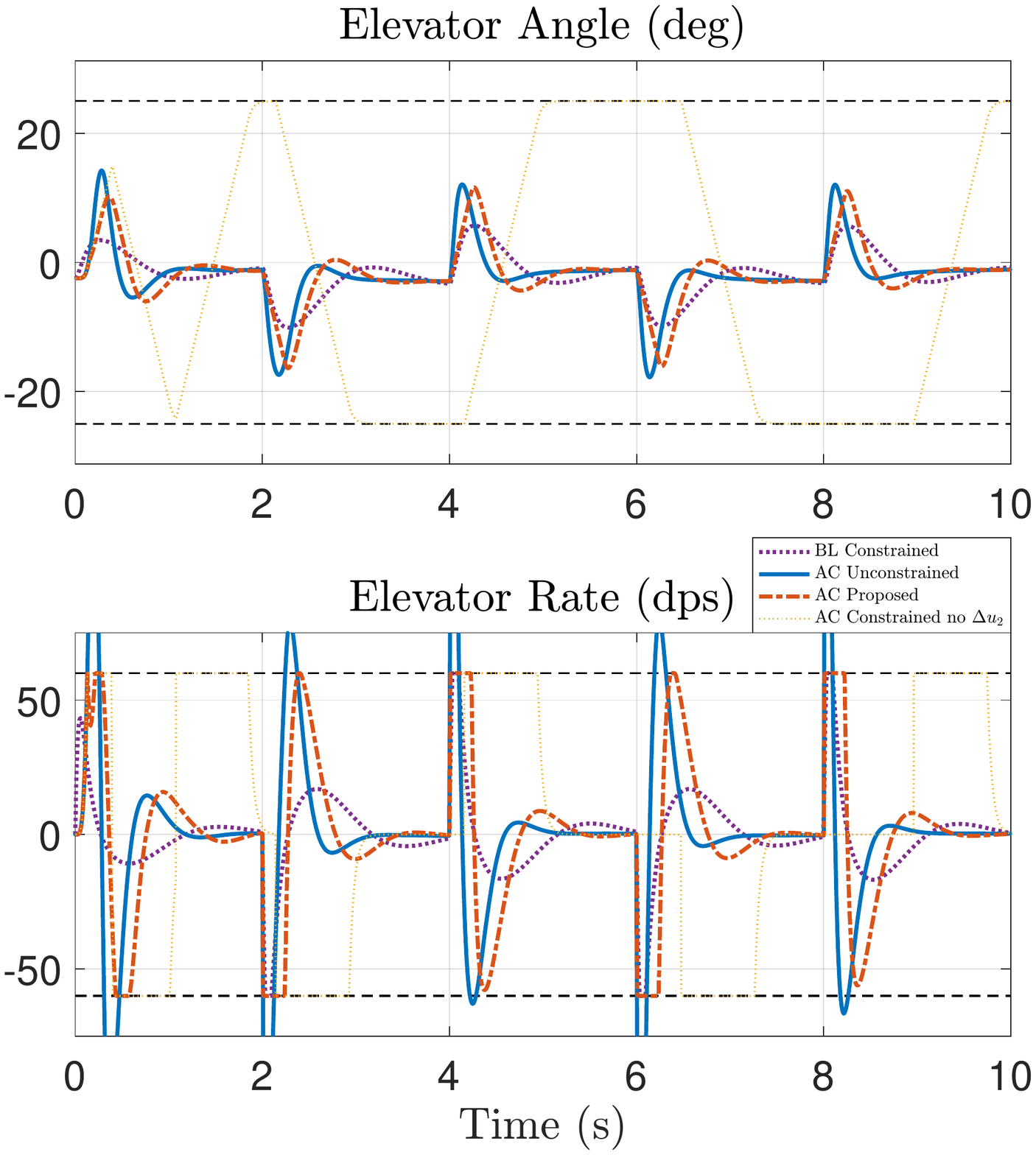}
	\caption{F-16 aircraft. Longitudinal control magnitude and rate response comparison between baseline only control (BL Constrained), unconstrained adaptive control (AC Unconstrained), and constrained adaptive control with (AC Proposed) and without (AC Constrained no $\Delta u_2$) update law saturation modification.}
	\label{f:control_1}
\end{figure}
\begin{figure}[!h]
	\centering
	\includegraphics[trim={2.0cm 9cm 2cm 9cm},clip,width=\linewidth]{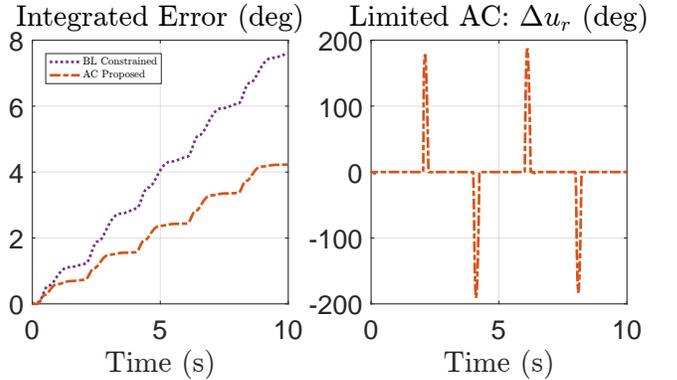}
	\caption{F-16 aircraft. Left: Integral of the absolute value of the error between the ideal (unconstrained) adaptive control pitch rate response and the constrained baseline/proposed adaptive responses of Figure \ref{f:state_2}. Right: Time history of rate saturation elevator disturbance (\ref{e:Control_Deficiency}) for the proposed constrained adaptive controller with saturation modification.}
	\label{f:long_error}
\end{figure}

In order to demonstrate the effectiveness of the proposed adaptive controller for a MIMO system, with both magnitude and rate limits encountered, similar simulations were carried out for the lateral-directional subsystem of the numerical nonlinear F-16 simulation in Figures \ref{f:state_3} and \ref{f:control_2}. In this simulation the aircraft experiences a $50\%$ reduction in rolling moment coefficient $C_{\ell}$. The uncertainty results in a degradation in the performance of the roll rate response for the non-adaptive baseline controller as compared to the ideal adaptive controller for the plant model without input limits, as seen in Figure \ref{f:latr_error}. The adaptive controller designed without the error model and update laws modified for the presence of input rate and magnitude limits can be seen to have a severely degraded response. In comparison, the proposed adaptive controller presented in this paper more closely recovers the ideal unconstrained response (up to the performance limits due to the input magnitude and rate constraints). In this simulation the effects of rate limits is even more apparent, where in Figure \ref{f:control_2} it can be more clearly noticed that the limited adaptive control aileron angle response evolves linearly when an aileron rate limit is encountered. A comparison of the responses between the limited baseline and proposed adaptive controllers, along with a plot of the saturation disturbance $\Delta u_2$ for the proposed adaptive controller is shown in Figure \ref{f:latr_error}.

\subsection{Hypersonic Vehicle Numerical Simulation}
\label{ss:Hypersonic}

This section provides numerical simulation results for an open loop unstable nonlinear hypersonic vehicle model from \cite{Wiese_2015,Wiese_2016,Wiese_2017} in order to further demonstrate the need for an adaptive controller capable of handling both rate limits and vehicle uncertainties. This nonlinear vehicle model also has aerodynamic forces and moments calculated from aerodynamic parameters based on flight condition. In order to obtain linearized dynamics for control design as in (\ref{e:plant}), a trim point was obtained at a straight and level flying condition at Mach 6 at an altitude of 80,000 ft. The control design for the hypersonic vehicle was separated into velocity, longitudinal and lateral-directional dynamics. In order to demonstrate uncertainty in the longitudinal dynamics, separate controllers were designed for the velocity and lateral-directional subsystems to maintain the trim condition in their respective subsystems. As in \cite{Wiese_2015,Wiese_2016,Wiese_2017}, the longitudinal subsystem is restated as
\begin{equation*}
x_p=
\begin{bmatrix}
\alpha&q
\end{bmatrix}^T
\>~
u_p=\delta_e
\>~
y_p=q
\>~
z_p=q
\end{equation*}
where the longitudinal state is composed of the vehicle's angle of attack $\alpha$ and pitch rate $q$. Pitch rate is both the measured and regulated variable. An elevator deflection $\delta_e$ represents the input to the dynamics. The elevator gain parameter $\tau$ and saturation levels are as follows:
\begin{equation*}
\tau=0.02
\>~
u_{max}=\pm 30~\text{deg}
\>~
u_{r,max}=\pm 100~\text{dps}.
\end{equation*}

\begin{figure}[!t]
	\centering
	\includegraphics[trim={1.7cm 3.8cm 1.9cm 4.1cm},clip,width=\linewidth]{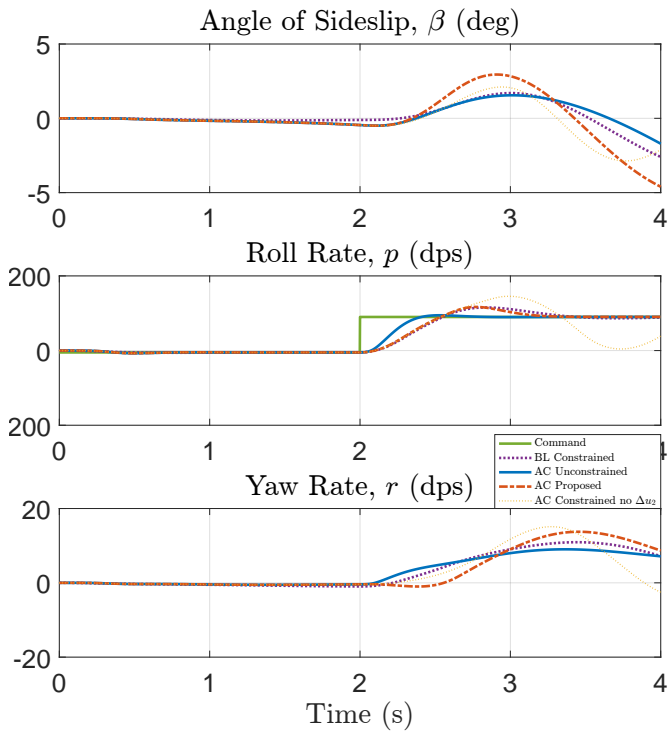}
	\caption{F-16 aircraft. Lateral-Directional state response comparison between baseline only control (BL Constrained), unconstrained adaptive control (AC Unconstrained), and constrained adaptive control with (AC Proposed) and without (AC Constrained no $\Delta u_2$) update law saturation modification.}
	\label{f:state_3}
\end{figure}
\begin{figure}[!h]
	\centering
	\includegraphics[trim={1.7cm 3.8cm 1.9cm 4.1cm},clip,width=\linewidth]{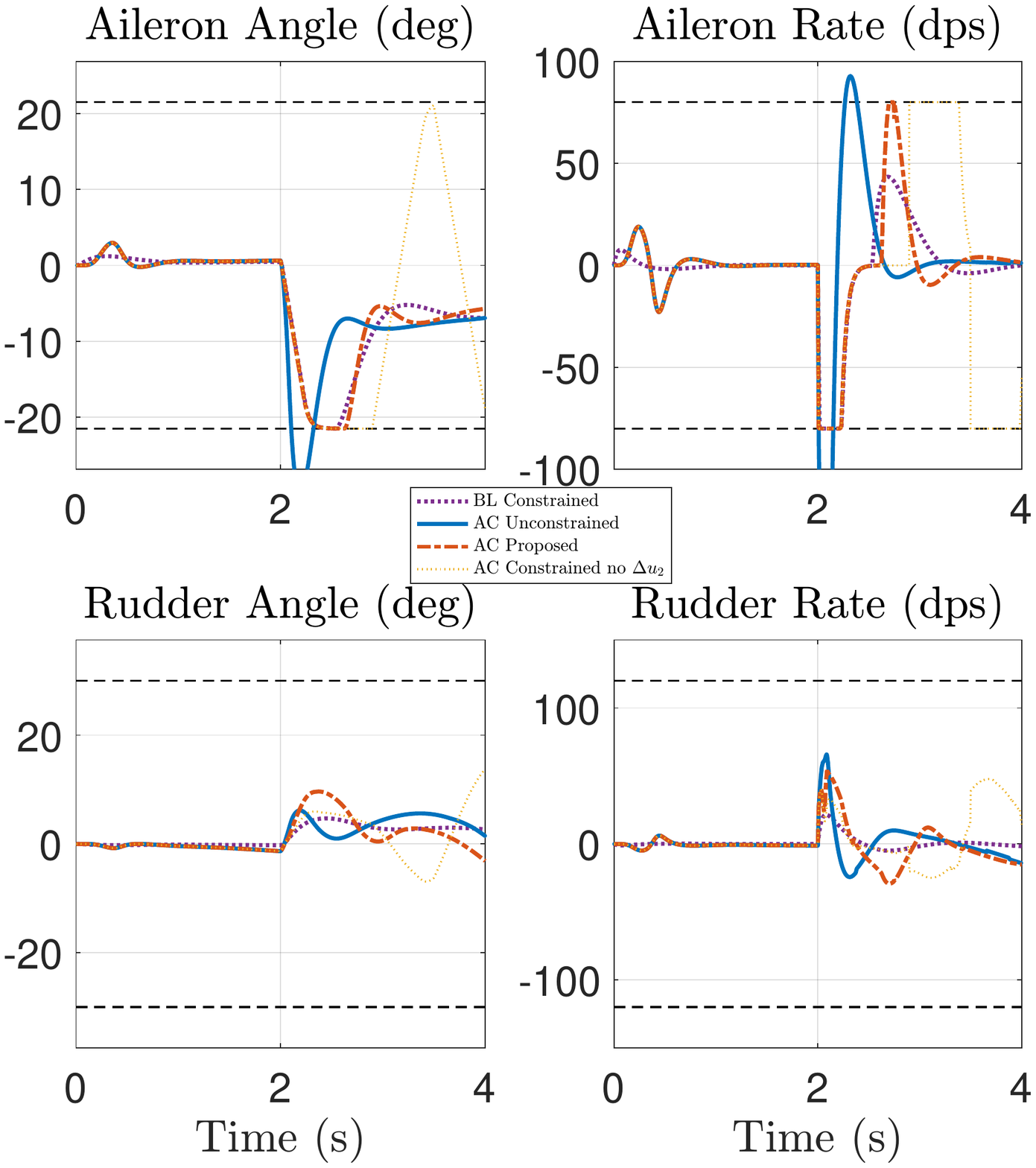}
	\caption{F-16 aircraft. Lateral-Directional control magnitude and rate response comparison between baseline only control (BL Constrained), unconstrained adaptive control (AC Unconstrained), and constrained adaptive control with (AC Proposed) and without (AC Constrained no $\Delta u_2$) update law saturation modification.}
	\label{f:control_2}
\end{figure}
\begin{figure}[!h]
	\centering
	\includegraphics[trim={2.0cm 9cm 2cm 9cm},clip,width=\linewidth]{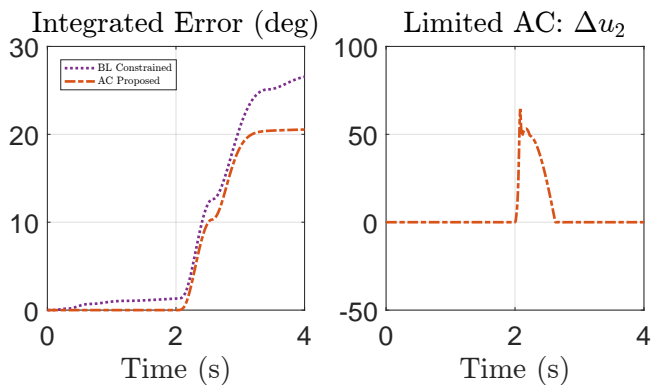}
	\caption{F-16 aircraft. Left: Integral of the absolute value of the error between the ideal (unconstrained) adaptive control pitch rate response and the constrained baseline/proposed adaptive responses of Figure \ref{f:state_3}. Right: Time history of combined magnitude and rate saturation aileron disturbance (\ref{e:Control_Deficiency}) for the proposed constrained adaptive controller with saturation modification.}
	\label{f:latr_error}
\end{figure}

In contrast to the F-16 simulation in the previous subsection, the longitudinal dynamics of this hypersonic vehicle numerical simulation are open loop unstable. Once again, integral command tracking (\ref{e:integral_error}) was included in addition to the explicit inclusion of the dynamics of the magnitude and rate saturated actuator as in (\ref{e:filter}), resulting in complete plant models as in (\ref{e:full_plant_full}), (\ref{e:Plant_Simple}) which satisfy Assumptions \ref{a:minimal}-\ref{a:rel1}, as required.

Figures \ref{f:state_2_hypersonic} and \ref{f:control_1_hypersonic} show the response of the hypersonic vehicle in the presence of a $50\%$ decrease in the control effectiveness of all control surfaces and a shift in the center of gravity $0.1$ ft rearward, with the pitch rate command stepped every two seconds. The uncertainty present in this simulation significantly degrades the ability of the non-adaptive baseline controller in (\ref{e:u_bl}) to track the desired pitch rate command. The angle of attack response of the baseline controller as seen in Figure \ref{f:state_2_hypersonic} additionally becomes unacceptably large for a hypersonic vehicle. This is another example of the motivation to employ an adaptive controller which is capable of countering the effects of uncertainty and restore the desired system response. In the absence of input limits, the adaptive control input, as in \cite{Zheng2013,Qu2015,Qu2016,Qu2016a}, and designed in (\ref{e:Control_Input}) can be seen to result satisfactory command tracking in Figure \ref{f:state_2_hypersonic} by the use of large control rates initially as seen in Figure \ref{f:control_1_hypersonic}. However, in the presence of input magnitude and rate limits, this controller can be seen to result in instability. The proposed adaptive controller which includes the modified error model (\ref{e:eu_dot}), adaptive update law (\ref{e:Omega_Delta_dot}) and explicitly accounts for the presence of input magnitude and rate limits can be seen to result in a stable response which recovers the desired command tracking performance. Figure \ref{f:long_error_hypersonic} shows the integrated absolute value of the error between the ideal unconstrained adaptive control response and each of the baseline response and the proposed adaptive control response. The baseline response can be seen to perform significantly worse as compared to the proposed adaptive control response which explicitly accounts for input limits. The open loop unstable nature of the longitudinal dynamics of the hypersonic vehicle can be seen to result in a greater difference in the performance of the adaptive input limiter controller as compared to the baseline controller.

\begin{figure}[!t]
	\centering
	\includegraphics[trim={1.7cm 3.8cm 1.9cm 4.1cm},clip,width=\linewidth]{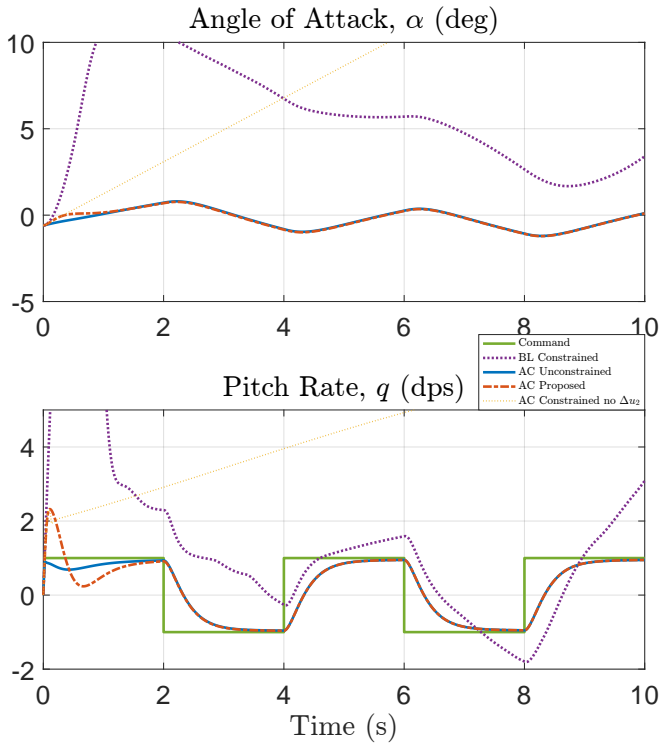}
	\caption{Hypersonic Vehicle. Longitudinal state response comparison between baseline only control (BL Constrained), unconstrained adaptive control (AC Unconstrained), and constrained adaptive control with (AC Proposed) and without (AC Constrained no $\Delta u_2$) update law saturation modification.}
	\label{f:state_2_hypersonic}
\end{figure}
\begin{figure}[!h]
	\centering
	\includegraphics[trim={1.7cm 3.8cm 1.9cm 4.1cm},clip,width=\linewidth]{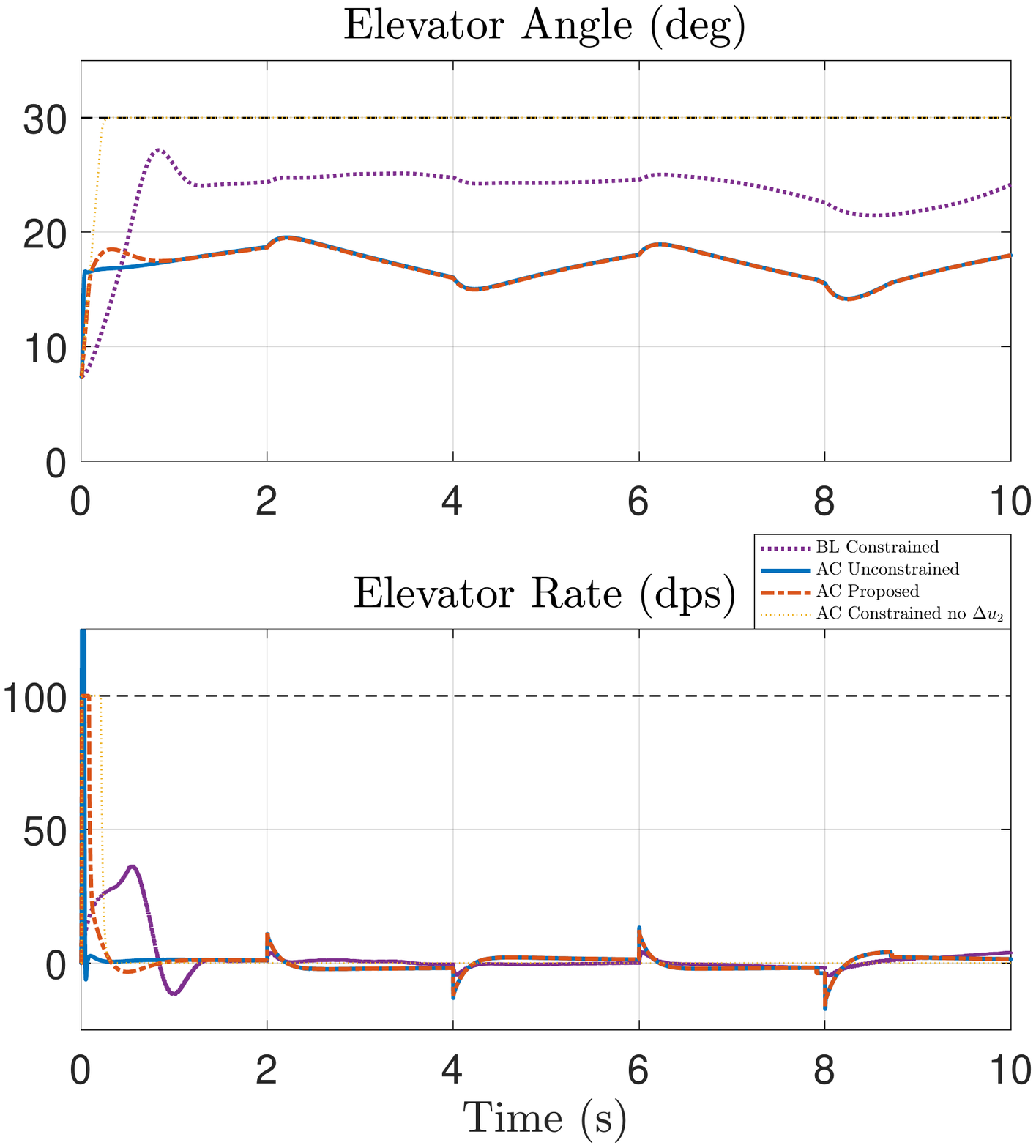}
	\caption{Hypersonic Vehicle. Longitudinal control magnitude and rate response comparison between baseline only control (BL Constrained), unconstrained adaptive control (AC Unconstrained), and constrained adaptive control with (AC Proposed) and without (AC Constrained no $\Delta u_2$) update law saturation modification.}
	\label{f:control_1_hypersonic}
\end{figure}
\begin{figure}[!h]
	\centering
	\includegraphics[trim={2.0cm 9cm 2cm 9cm},clip,width=\linewidth]{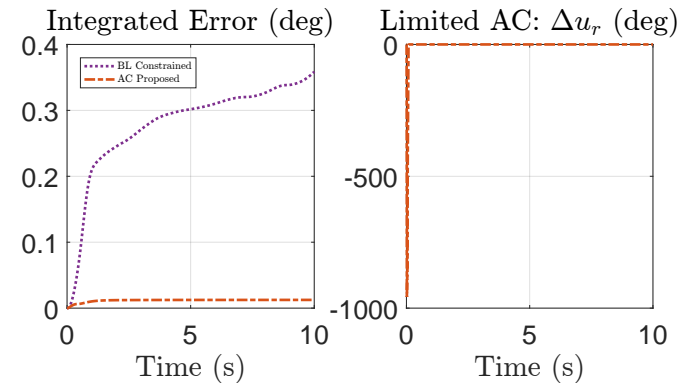}
	\caption{Hypersonic Vehicle. Left: Integral of the absolute value of the error between the ideal (unconstrained) adaptive control pitch rate response and the constrained baseline/proposed adaptive responses of Figure \ref{f:state_2_hypersonic}. Right: Time history of rate saturation elevator disturbance (\ref{e:Control_Deficiency}) for the proposed constrained adaptive controller with saturation modification.}
	\label{f:long_error_hypersonic}
\end{figure}

\section{Conclusion}\label{s:Conclusion}
This paper presented the first MIMO adaptive architecture for controlling a plant model in the presence of parametric uncertainties in the presence of hard limits on both input magnitude and rate. The plant input is limited in both magnitude and rate through the use of a filter with hard saturation nonlinearities placed in the control path. Augmented adaptive update laws which explicitly account for the presence of the input limits were derived to stabilize the system while in saturation. The proof of stability of signals in the closed-loop system contained two parts: boundedness of the adaptive parameters and boundedness of the state. The main stability result greatly extends the current state-of-the-art in MIMO adaptive control in the presence of rate limits, as it provides a solution in output feedback form with hard saturation nonlinearities employed. The proposed controller was applied to a nonlinear numerical simulation of an open loop stable F-16 aircraft as well as a nonlinear numerical simulation of an open loop unstable hypersonic vehicle. The MIMO simulations demonstrate successful input limiting, stability, and tracking performance of the adaptive controller presented in this paper.


%

\appendices
\section{Constants and Time-Varying Scalars}\label{app:constants}
Throughout this paper the norm symbol $||X||$ denotes the 2-norm of element $X$.
\subsection{Constant Scalars}
Define $P_B$ such that: $||\chi^T(t)P_{cl}B_{\Omega}||\leq P_B||\chi(t)||$

Define $P_C$ such that: $||\chi^T(t)P_{cl}C_{\Delta \bar{u}_2}||\leq P_C||\chi(t)||$

Define $P_Z$ such that: $||\chi^T(t)P_{cl}B_z||\leq P_Z||\chi(t)||$
$$\gamma_{max}=\max\left[eig(\Gamma_{\Omega}),eig(\Gamma_{\Omega_{\Delta}})\right],\quad \lambda_{min}=\min(eig(\Lambda^*))$$
$$u_{min}=\min_i(u_{max,i}),\quad u_{r,min}=\min_i(u_{r,max,i})$$
$$p_{min}=\min(eig(P_{cl})),\quad p_{max}=\max(eig(P_{cl}))$$
$$\rho=\sqrt{\frac{p_{max}}{p_{min}}},\quad q_0=\min(eig(Q_{cl}))$$
$$\tilde{\Omega}_{max}=\sup_t||\tilde{\Omega}(t)||$$
$$||B_{\xi,\Omega}||=(||\Omega^*||+\tilde{\Omega}_{max})||B_{\xi}||$$
$$||K_{\xi,\Omega}||=(||\Omega^*||+\tilde{\Omega}_{max})||K_{\xi}||$$
\begin{align*}
\begin{split}
||&K_{u_p}||=\\
&||C_{u_p}||\left(||B_2\Lambda^*a_1^1||\left(\tilde{\Omega}_{max}||C_{\bar{\xi}}||+\left(||\Omega^*||+1\right)\right)+2\right)
\end{split}
\end{align*}
\begin{equation}
\kappa_1=2P_C||S_C||\tilde{\Omega}_{max}||\Gamma_C||
\end{equation}
\begin{equation}
\kappa_2=\left|-q_0+2\tilde{\Omega}_{max}P_B||C_{\bar{\xi}}||+2P_C\frac{1}{a_1^1}||K_{\xi,\Omega}||\right|
\end{equation}
\begin{equation}
\kappa_3=\alpha u_{min}-\left(2P_Z+2P_C\frac{1}{a_1^1}||B_{\xi,\Omega}||\right)z_{cmd,max}
\end{equation}
\begin{equation}
\kappa_4=\tilde{\Omega}_{max}^2\frac{a_1^1P_B||S_C||\cdot||\Gamma_C||\cdot||C_{\bar{\xi}}||}{||K_{\xi,\Omega}||}
\end{equation}
\begin{equation}
\kappa_5=q_0-3\tilde{\Omega}_{max}P_B||C_{\bar{\xi}}||
\end{equation}
\begin{equation}
\kappa_6=\left(2P_Z+\tilde{\Omega}_{max}\frac{P_B||B_{\xi,\Omega}||\cdot||C_{\bar{\xi}}||}{||K_{\xi,\Omega}||}\right)z_{cmd,max}
\end{equation}
\begin{equation}
\kappa_7=\left|-q_0+2\tilde{\Omega}_{max}P_B||C_{\bar{\xi}}||+2P_C\frac{1}{a_1^1}\left(||K_{\xi,\Omega}||+||K_{u_p}||\right)\right|
\end{equation}
\begin{equation}
\kappa_8=\beta u_{r,min}-\left(2P_Z+2P_C\frac{1}{a_1^1}||B_{\xi,\Omega}||\right)z_{cmd,max}
\end{equation}
\begin{equation}
\kappa_9=\tilde{\Omega}_{max}^2\frac{a_1^1P_B||S_C||\cdot||\Gamma_C||\cdot||C_{\bar{\xi}}||}{||K_{\xi,\Omega}||+||K_{u_p}||}
\end{equation}
\begin{equation}
\kappa_{10}=\left(2P_Z+\tilde{\Omega}_{max}\frac{P_B||B_{\xi,\Omega}||\cdot||C_{\bar{\xi}}||}{||K_{\xi,\Omega}||+||K_{u_p}||}\right)z_{cmd,max}
\end{equation}
\begin{equation}
\kappa_{11}=2P_Zz_{cmd,max}
\end{equation}
\begin{equation}
\kappa_{12}=q_0-2\tilde{\Omega}_{max}P_B||C_{\bar{\xi}}||
\end{equation}
\begin{align}\label{e:chi_min}
\begin{split}
\chi_{min}=\max&\left(\frac{\kappa_5-\sqrt{\kappa_5^2-4\kappa_4\kappa_6}}{2\kappa_4},\right.\\
&\left.\frac{\kappa_5-\sqrt{\kappa_5^2-4\kappa_9\kappa_{10}}}{2\kappa_9},\frac{\kappa_{11}}{\kappa_{12}}\right)
\end{split}
\end{align}
\begin{align}
\begin{split}
\label{e:chi_max}
\chi_{max}=\min&\left(\frac{\sqrt{\kappa_2^2+4\kappa_1\kappa_3}-\kappa_2}{2\kappa_1},\frac{\kappa_5+\sqrt{\kappa_5^2-4\kappa_4\kappa_6}}{2\kappa_4},\right.\\
&\left.\frac{\sqrt{\kappa_7^2+4\kappa_1\kappa_8}-\kappa_7}{2\kappa_1},\frac{\kappa_5+\sqrt{\kappa_5^2-4\kappa_9\kappa_{10}}}{2\kappa_9}\right)
\end{split}
\end{align}
\begin{equation}
\bar{\Omega}_{max}=\frac{q_0}{3P_B||C_{\bar{\xi}}||}
\end{equation}
Note that $\tilde{\Omega}_{max}<\sqrt{\frac{V(0)\gamma_{max}}{\lambda_{min}}}$. Additionally, the ratio $\frac{q_0}{p_{max}}$ is maximized with a choice of $Q_{cl}=I$ as seen in reference \cite{Slotine1991}. A bound on $\tilde{\Omega}$ may be enforced using the projection operator as is common in adaptive control \cite{Hussain_2017}.

\subsection{Constant Matrices}
$$A_{cl}=
\begin{bmatrix}
-\frac{a_1^0}{a_1^1}I&\frac{1}{a_1^1}I&0&0\\
0&A-B_2K&LC&0\\
0&0&A^*_L&B_2^1\Lambda^*\\
0&0&0&-\frac{a_1^0}{a_1^1}I
\end{bmatrix}$$
$$B_{\Omega}^T=
\begin{bmatrix}
0&0&(B_2^1\Lambda^*)^T&0
\end{bmatrix}$$
$$B_Z^T=
\begin{bmatrix}
0&B_z^T&0&0
\end{bmatrix}, \quad
B_{\xi}^T=
\begin{bmatrix}
0&a_1^1B_z^T&0
\end{bmatrix}$$
$$C_{\Delta \bar{u}_2}^T=
\begin{bmatrix}
0&0&0&I
\end{bmatrix},\quad
C_{u_p}=
\begin{bmatrix}
0&\Lambda^{*-1}&0
\end{bmatrix}$$
$$C_{\bar{\xi}}=
\begin{bmatrix}
K&0&0&0\\
0&I&0&0\\
I&0&0&0
\end{bmatrix}$$
$$K_{\xi}=
\begin{bmatrix}
0&K&0&0\\
0&(a_1^0I+a_1^1(A-B_2K))&a_1^1LC&0\\
0&I&0&0
\end{bmatrix}$$
$$S_C=S_2C,\quad \Gamma_C=C_{\bar{\xi}}^T\Gamma_{\Omega}C_{\bar{\xi}}$$

\subsection{Time-Varying Scalars}
For ease of exposition, the following time varying scalars which represent the ratio of the saturated input magnitude and rate respectively to the unsaturated input magnitude and rate:
$$\mathbb{U}=\frac{||E_s(u(t),u_{max})||}{||u(t)||},\quad \mathbb{U}_r=\frac{||E_s(u_r(t),u_{r,max})||}{||u_r(t)||}$$
The magnitude of the computed control input (\ref{e:Control_Input}) may be expressed as follows:
\begin{equation}
\label{e:norm_u}
\begin{alignedat}{2}
||u(t)||&=&~&||(a^1_1s+a^0_1)\Omega^T(t)\bar{\xi}(t)||\\
&=&~&||\Omega^T(t)(a_1^1s+a_1^0)\bar{\xi}(t)+a_1^1\dot{\Omega}^T(t)\bar{\xi}(t)||\\
&=&~&||-\Omega^T(t)K_{\xi}\chi(t)-\Omega^T(t)B_{\xi}z_{cmd}(t)\\
&&~&-a_1^1S_2Ce_{mu}(t)\chi^T(t)\Gamma_C\chi(t)||\\
&\leq&~&||K_{\xi,\Omega}||\cdot||\chi(t)||+||B_{\xi,\Omega}||z_{cmd,max}\\
&&&+a_1^1||S_C||\tilde{\Omega}_{max}||\Gamma_C||\cdot||\chi(t)||^2
\end{alignedat}
\end{equation}
The magnitude of the plant input may be expressed as follows, through applying equation (\ref{e:e_mx_representation}):
\begin{equation}
\label{e:norm_u_p}
\begin{alignedat}{2}
||u_p(t)||&=&~&||C_{u_p}(e_x(t)+x_m(t))||\\
&=&~&||C_{u_p}(B_2\Lambda^*a_1^1\left[\bar{\Psi}^{*T}_2\bar{x}_m(t)+\Omega^T(t)\bar{\xi}(t)\right.\\
&&&\left.-\Lambda^{*-1}\bar{u}_{bl}(t)+\Delta \bar{u}_2(t)\right]+e_{mx}(t)+x_m(t))||\\
&\leq&~&||K_{u_p}||\cdot||\chi(t)||
\end{alignedat}
\end{equation}
Additionally, from equation (\ref{e:u_r}) the following inequality holds:
\begin{equation}
\label{e:norm_u_r}
||u_r(t)||\leq\frac{1}{\tau}\left(||u(t)||+||u_p(t)||\right)
\end{equation}

\section{Proof of Theorem \ref{t:One}, Step 2}\label{app:proof}
\subsection{$\Delta u=0,\Delta u_r=0$}
In this case (\ref{e:Chi_full}) can be simplified as follows:
\begin{equation*}
\dot{\chi}(t)=A_{cl}\chi(t)-B_{\Omega}\tilde{\Omega}^{T}(t)C_{\bar{\xi}}\chi(t)+B_Zz_{cmd}(t)
\end{equation*}
This leads to the following time derivative of the Lyapunov function in equation (\ref{e:Lyap}):
\begin{equation*}
\begin{alignedat}{2}
\dot{W}(\chi(t))&=&~&-\chi^T(t)Q_{cl}\chi(t)-2\chi^T(t)P_{cl}B_{\Omega}\tilde{\Omega}^{T}(t)C_{\bar{\xi}}\chi(t)\\
&&&+2\chi^T(t)P_{cl}B_Zz_{cmd}(t)
\end{alignedat}
\end{equation*}
The right hand side may be bounded as follows:
\begin{equation*}
\begin{alignedat}{2}
\dot{W}(\chi(t))&\leq&~&-\left(q_0-2\tilde{\Omega}_{max}P_B||C_{\bar{\xi}}||\right)||\chi(t)||^2\\
&&&+\left(2P_Zz_{cmd,max}\right)||\chi(t)||
\end{alignedat}
\end{equation*}
Using the definition of $\tilde{\Omega}_{max}$ and condition (\textit{ii}) of Theorem \ref{t:One}, implies the following:
\begin{equation*}
\dot{W}(\chi(t))<0,\quad ||\chi(t)||>\frac{\kappa_{11}}{\kappa_{12}}
\end{equation*}
Choosing $\chi_{min}$ from equation (\ref{e:chi_min}) implies the following:
\begin{equation}
\label{e:1}
\dot{W}(\chi(t))<0,\quad \forall\chi(t)\in \mathcal{A}\text{ in case 1}
\end{equation}

\subsection{$\Delta u\neq0,\Delta u_r=0$}
In this case (\ref{e:Chi_full}) can be rewritten as:
\begin{equation*}
\begin{alignedat}{2}
\dot{\chi}(t)&=&~&A_{cl}\chi(t)-B_{\Omega}\tilde{\Omega}^{T}(t)C_{\bar{\xi}}\chi(t)\\
&&&+C_{\Delta \bar{u}_2}\frac{1}{a_1^1}\left(\mathbb{U}-1\right)u(t)+B_Zz_{cmd}(t)
\end{alignedat}
\end{equation*}
This leads to the following time derivative of the Lyapunov function in equation (\ref{e:Lyap}):
\begin{equation*}
\begin{alignedat}{2}
\dot{W}(\chi(t))&=&~&-\chi^T(t)Q_{cl}\chi(t)-2\chi^T(t)P_{cl}B_{\Omega}\tilde{\Omega}^{T}(t)C_{\bar{\xi}}\chi(t)\\
&&&+2\chi^T(t)P_{cl}C_{\Delta \bar{u}_2}\frac{1}{a_1^1}\left(\mathbb{U}-1\right)u(t)\\
&&&+2\chi^T(t)P_{cl}B_Zz_{cmd}(t)
\end{alignedat}
\end{equation*}
Three sub-cases are considered in this section.

\subsubsection{Sub-case i}
\begin{equation*}
2\chi^T(t)P_{cl}C_{\Delta \bar{u}_2}\frac{1}{a_1^1}\mathbb{U}u(t)<-\alpha u_{min}||\chi(t)||
\end{equation*}
The time derivative of the Lyapunov function may be bounded as follows using the condition of this sub-case and (\ref{e:norm_u}):
\begin{align*}
\begin{split}
&\dot{W}(\chi(t))\leq\left(2P_C||S_C||\tilde{\Omega}_{max}||\Gamma_C||\right)||\chi(t)||^3\\
&+\left(-q_0+2\tilde{\Omega}_{max}P_B||C_{\bar{\xi}}||+2P_C\frac{1}{a_1^1}||K_{\xi,\Omega}||\right)||\chi(t)||^2\\
&-\left(\alpha u_{min}-\left(2P_Z+2P_C\frac{1}{a_1^1}||B_{\xi,\Omega}||\right)z_{cmd,max}\right)||\chi(t)||\\
\end{split}
\end{align*}
Thus:
\begin{equation*}
\dot{W}(\chi(t))<0,\quad ||\chi(t)||<\frac{\sqrt{\kappa_2^2+4\kappa_1\kappa_3}-\kappa_2}{2\kappa_1}
\end{equation*}
Choosing $\chi_{max}$ from equation (\ref{e:chi_max}) implies the following:
\begin{equation}
\label{e:21}
\dot{W}(\chi(t))<0,\quad \forall\chi(t)\in \mathcal{A}\text{ in case 2, sub-case (\textit{i})}
\end{equation}

\subsubsection{Sub-case ii}
\begin{equation*}
0>2\chi^T(t)P_{cl}C_{\Delta \bar{u}_2}\frac{1}{a_1^1}\mathbb{U}u(t)>-\alpha u_{min}||\chi(t)||
\end{equation*}
The condition of this sub-case implies the following:
\begin{equation*}
0<2\chi^T(t)P_{cl}C_{\Delta \bar{u}_2}\frac{1}{a_1^1}\mathbb{U}u(t)+\alpha u_{min}||\chi(t)||
\end{equation*}
Apply $\frac{u_{min}}{||E_s(u(t),u_{max})||}\leq 1$:
\begin{equation*}
0<2\chi^T(t)P_{cl}C_{\Delta \bar{u}_2}\frac{1}{a_1^1}u(t)+\alpha||u(t)||\cdot||\chi(t)||
\end{equation*}
Add terms to create $\dot{W}(\chi(t))$ on the left hand side:
\begin{equation*}
\begin{alignedat}{2}
\dot{W}(\chi(t))&<&~&-\chi^T(t)Q_{cl}\chi(t)-2\chi^T(t)P_{cl}B_{\Omega}\tilde{\Omega}^{T}(t)C_{\bar{\xi}}\chi(t)\\
&&&+2\chi^T(t)P_{cl}C_{\Delta \bar{u}_2}\frac{1}{a_1^1}\left(\mathbb{U}-1\right)u(t)\\
&&&+2\chi^T(t)P_{cl}B_Zz_{cmd}(t)\\
&&&+2\chi^T(t)P_{cl}C_{\Delta \bar{u}_2}\frac{1}{a_1^1}u(t)\\
&&&+\alpha||u(t)||\cdot||\chi(t)||
\end{alignedat}
\end{equation*}
Apply case inequality:
\begin{equation*}
0>2\chi^T(t)P_{cl}C_{\Delta \bar{u}_2}\frac{1}{a_1^1}\mathbb{U}u(t)
\end{equation*}
Thus:
\begin{equation*}
\begin{alignedat}{2}
\dot{W}(\chi(t))&<&~&-\chi^T(t)Q_{cl}\chi(t)-2\chi^T(t)P_{cl}B_{\Omega}\tilde{\Omega}^T(t)C_{\bar{\xi}}\chi(t)\\
&&&+\alpha||u(t)||\cdot||\chi(t)||+2\chi^T(t)P_{cl}B_Zz_{cmd}(t)
\end{alignedat}
\end{equation*}
Apply (\ref{e:norm_u}):
\begin{equation*}
\begin{alignedat}{2}
\dot{W}(\chi(t))&<&~&-\chi^T(t)Q_{cl}\chi(t)-2\chi^T(t)P_{cl}B_{\Omega}\tilde{\Omega}^T(t)C_{\bar{\xi}}\chi(t)\\
&&&+\alpha ||K_{\xi,\Omega}||\cdot||\chi(t)||^2\\
&&&+\alpha ||B_{\xi,\Omega}||z_{cmd,max}||\chi(t)||\\
&&&+\alpha a_1^1||S_C||\tilde{\Omega}_{max}||\Gamma_C||\cdot||\chi(t)||^3\\
&&&+2\chi^T(t)P_{cl}B_Zz_{cmd}(t)
\end{alignedat}
\end{equation*}
Maximize the right hand side:
\begin{equation*}
\begin{alignedat}{2}
\dot{W}(\chi(t))&<&~&\left(\alpha a_1^1||S_C||\tilde{\Omega}_{max}||\Gamma_C||\right)||\chi(t)||^3\\
&&&-\left(q_0-2P_B\tilde{\Omega}_{max}||C_{\bar{\xi}}||-\alpha ||K_{\xi,\Omega}||\right)||\chi(t)||^2\\
&&&+\left((2P_Z+\alpha ||B_{\xi,\Omega}||)z_{cmd,max}\right)||\chi(t)||
\end{alignedat}
\end{equation*}
Apply $\alpha=\frac{P_B\tilde{\Omega}_{max}||C_{\bar{\xi}}||}{||K_{\xi,\Omega}||}$:
\begin{equation*}
\begin{alignedat}{2}
\dot{W}(\chi(t))&<&~&\left(\tilde{\Omega}_{max}^2\frac{a_1^1P_B||S_C||\cdot||\Gamma_C||\cdot||C_{\bar{\xi}}||}{||K_{\xi,\Omega}||}\right)||\chi(t)||^3\\
&&&-\left(q_0-3\tilde{\Omega}_{max}P_B||C_{\bar{\xi}}||\right)||\chi(t)||^2\\
&&&+\left(\left(2P_Z+\tilde{\Omega}_{max}\frac{P_B||B_{\xi,\Omega}||\cdot||C_{\bar{\xi}}||}{||K_{\xi,\Omega}||}\right)\right.\\
&&&\left.\times z_{cmd,max}\right)||\chi(t)||
\end{alignedat}
\end{equation*}
Using the definition of $\tilde{\Omega}_{max}$ and condition (\textit{ii}) of Theorem \ref{t:One}, implies the following:
\begin{align*}
\begin{split}
&\dot{W}(\chi(t))<0,\quad\text{for}\\ &\frac{\kappa_5-\sqrt{\kappa_5^2-4\kappa_4\kappa_6}}{2\kappa_4}<||\chi(t)||<\frac{\kappa_5+\sqrt{\kappa_5^2-4\kappa_4\kappa_6}}{2\kappa_4}
\end{split}
\end{align*}
Choosing $\chi_{min},\chi_{max}$ from equations (\ref{e:chi_min}) and (\ref{e:chi_max}) implies the following:
\begin{equation}
\label{e:22}
\dot{W}(\chi(t))<0,\quad \forall\chi(t)\in \mathcal{A}\text{ in case 2, sub-case (\textit{ii})}
\end{equation}

\subsubsection{Sub-case iii}
\begin{equation*}
2\chi^T(t)P_{cl}C_{\Delta \bar{u}_2}\frac{1}{a_1^1}\mathbb{U}u(t)>0
\end{equation*}
The condition of this sub-case implies the following inequality:
\begin{equation*}
2\chi^T(t)P_{cl}C_{\Delta \bar{u}_2}\frac{1}{a_1^1}\mathbb{U}u(t)<2\chi^T(t)P_{cl}C_{\Delta \bar{u}_2}\frac{1}{a_1^1}u(t)
\end{equation*}
Add terms to create $\dot{W}(\chi(t))$ on the left hand side:
\begin{equation*}
\begin{alignedat}{2}
\dot{W}(\chi(t))&\leq&~&-\chi^T(t)Q_{cl}\chi(t)-2\chi^T(t)P_{cl}B_{\Omega}\tilde{\Omega}^{T}(t)C_{\bar{\xi}}\chi(t)\\
&&&+2\chi^T(t)P_{cl}B_Zz_{cmd}(t)
\end{alignedat}
\end{equation*}
The right hand side may be maximized as:
\begin{equation*}
\begin{alignedat}{2}
\dot{W}(\chi(t))&\leq&~&-\left(q_0-2\tilde{\Omega}_{max}P_B||C_{\bar{\xi}}||\right)||\chi(t)||^2\\
&&&+\left(2P_Zz_{cmd,max}\right)||\chi(t)||
\end{alignedat}
\end{equation*}
Using the definition of $\tilde{\Omega}_{max}$ and condition (\textit{ii}) of Theorem \ref{t:One}, implies the following:
\begin{equation*}
\dot{W}(\chi(t))<0,\quad ||\chi(t)||>\frac{\kappa_{11}}{\kappa_{12}}
\end{equation*}
Choosing $\chi_{min}$ from equation (\ref{e:chi_min}) implies the following:
\begin{equation}
\label{e:23}
\dot{W}(\chi(t))<0,\quad \forall\chi(t)\in \mathcal{A}\text{ in case 2, sub-case (\textit{iii})}
\end{equation}

\subsection{$\Delta u=0,\Delta u_r\neq0$}
In this case (\ref{e:Chi_full}) can be rewritten as:
\begin{equation*}
\begin{alignedat}{2}
\dot{\chi}(t)&=&~&A_{cl}\chi(t)-B_{\Omega}\tilde{\Omega}^{T}(t)C_{\bar{\xi}}\chi(t)\\
&&&+C_{\Delta \bar{u}_2}\frac{1}{a_1^1}\left(\mathbb{U}_r-1\right)\tau u_r(t)+B_Zz_{cmd}(t)
\end{alignedat}
\end{equation*}
The following is the time derivative of the Lyapunov function in equation (\ref{e:Lyap}):
\begin{equation*}
\begin{alignedat}{2}
\dot{W}(\chi(t))&=&~&-\chi^T(t)Q_{cl}\chi(t)-2\chi^T(t)P_{cl}B_{\Omega}\tilde{\Omega}^{T}(t)C_{\bar{\xi}}\chi(t)\\
&&&+2\chi^T(t)P_{cl}C_{\Delta \bar{u}_2}\frac{1}{a_1^1}\left(\mathbb{U}_r-1\right)\tau u_r(t)\\
&&&+2\chi^T(t)P_{cl}B_Zz_{cmd}(t)
\end{alignedat}
\end{equation*}
Three sub-cases are considered in this section.

\subsubsection{Sub-case i}
\begin{equation*}
2\chi^T(t)P_{cl}C_{\Delta \bar{u}_2}\frac{1}{a_1^1}\mathbb{U}_r\tau u_r(t)<-\beta u_{r,min}||\chi(t)||
\end{equation*}
The time derivative of the Lyapunov function may be bounded as follows using the condition of this sub-case and (\ref{e:norm_u_r}):
\begin{align*}
\begin{split}
&\dot{W}(\chi(t))\leq\left(2P_C||S_C||\tilde{\Omega}_{max}||\Gamma_C||\right)||\chi(t)||^3\\
&+\left(-q_0+2\tilde{\Omega}_{max}P_B||C_{\bar{\xi}}||\right.\\
&\left.+2P_C\frac{1}{a_1^1}\left(||K_{\xi,\Omega}||+||K_{u_p}||\right)\right)||\chi(t)||^2\\
&-\left(\beta u_{r,min}-\left(2P_Z+2P_C\frac{1}{a_1^1}||B_{\xi,\Omega}||\right)z_{cmd,max}\right)||\chi(t)||
\end{split}
\end{align*}
Thus:
\begin{equation*}
\dot{W}(\chi(t))<0,\quad ||\chi(t)||<\frac{\sqrt{\kappa_7^2+4\kappa_1\kappa_8}-\kappa_7}{2\kappa_1}
\end{equation*}
Choosing $\chi_{max}$ from equation (\ref{e:chi_max}) implies the following:
\begin{equation}
\label{e:31}
\dot{W}(\chi(t))<0,\quad \forall\chi(t)\in \mathcal{A}\text{ in case 3, sub-case (\textit{i})}
\end{equation}

\subsubsection{Sub-case ii}
\begin{equation*}
0>2\chi^T(t)P_{cl}C_{\Delta \bar{u}_2}\frac{1}{a_1^1}\mathbb{U}_r\tau u_r(t)>-\beta u_{r,min}||\chi(t)||
\end{equation*}
The condition of this sub-case implies the following:
\begin{equation*}
0<2\chi^T(t)P_{cl}C_{\Delta \bar{u}_2}\frac{1}{a_1^1}\mathbb{U}_r\tau u_r(t)+\beta u_{r,min}||\chi(t)||
\end{equation*}
Apply $\frac{u_{r,min}}{||E_s(u_r(t),u_{r,max})||}\leq 1$:
\begin{equation*}
0<2\chi^T(t)P_{cl}C_{\Delta \bar{u}_2}\frac{1}{a_1^1}\tau u_r(t)+\beta||u_r(t)||\cdot||\chi(t)||
\end{equation*}
Add terms to create $\dot{W}(\chi(t))$ on the left hand side:
\begin{equation*}
\begin{alignedat}{2}
\dot{W}(\chi(t))&<&~&-\chi^T(t)Q_{cl}\chi(t)-2\chi^T(t)P_{cl}B_{\Omega}\tilde{\Omega}^{T}(t)C_{\bar{\xi}}\chi(t)\\
&&&+2\chi^T(t)P_{cl}C_{\Delta \bar{u}_2}\frac{1}{a_1^1}\left(\mathbb{U}_r-1\right)\tau u_r(t)\\
&&&+2\chi^T(t)P_{cl}B_Zz_{cmd}(t)\\
&&&+2\chi^T(t)P_{cl}C_{\Delta \bar{u}_2}\frac{1}{a_1^1}\tau u_r(t)\\
&&&+\beta||u_r(t)||\cdot||\chi(t)||
\end{alignedat}
\end{equation*}
Apply case inequality:
\begin{align*}
\begin{split}
&0>2\chi^T(t)P_{cl}C_{\Delta \bar{u}_2}\frac{1}{a_1^1}\mathbb{U}_r\tau u_r(t)
\end{split}
\end{align*}
Thus:
\begin{equation*}
\begin{alignedat}{2}
\dot{W}(\chi(t))&<&~&-\chi^T(t)Q_{cl}\chi(t)-2\chi^T(t)P_{cl}B_{\Omega}\tilde{\Omega}^{T}(t)C_{\bar{\xi}}\chi(t)\\
&&&+\beta||u_r(t)||\cdot||\chi(t)||+2\chi^T(t)P_{cl}B_Zz_{cmd}(t)
\end{alignedat}
\end{equation*}
Apply (\ref{e:norm_u}), (\ref{e:norm_u_p}), and (\ref{e:norm_u_r}):
\begin{equation*}
\begin{alignedat}{2}
\dot{W}(\chi(t))&<&~&-\chi^T(t)Q_{cl}\chi(t)-2\chi^T(t)P_{cl}B_{\Omega}\tilde{\Omega}^T(t)C_{\bar{\xi}}\chi(t)\\
&&&+\frac{\beta}{\tau} \left(||K_{\xi,\Omega}||+||K_{u_p}||\right)||\chi(t)||^2\\
&&&+\frac{\beta}{\tau} ||B_{\xi,\Omega}||z_{cmd,max}||\chi(t)||\\
&&&+\frac{\beta}{\tau} a_1^1||S_C||\tilde{\Omega}_{max}||\Gamma_C||\cdot||\chi(t)||^3\\
&&&+2\chi^T(t)P_{cl}B_Zz_{cmd}(t)
\end{alignedat}
\end{equation*}
Maximize the right hand side:
\begin{equation*}
\begin{alignedat}{2}
\dot{W}(\chi(t))&<&~&\left(\frac{\beta}{\tau} a_1^1||S_C||\tilde{\Omega}_{max}||\Gamma_C||\right)||\chi(t)||^3\\
&&&+\left(-q_0+2P_B\tilde{\Omega}_{max}||C_{\bar{\xi}}||\right.\\
&&&\left.+\frac{\beta}{\tau} \left(||K_{\xi,\Omega}||+||K_{u_p}||\right)\right)||\chi(t)||^2\\
&&&+\left(\left(2P_Z+\frac{\beta}{\tau}||B_{\xi,\Omega}||\right)z_{cmd,max}\right)||\chi(t)||
\end{alignedat}
\end{equation*}
Apply $\beta=\frac{\tau P_B\tilde{\Omega}_{max}||C_{\bar{\xi}}||}{||K_{\xi,\Omega}||+||K_{u_p}||}$:
\begin{equation*}
\begin{alignedat}{2}
\dot{W}(\chi(t))&<&~&\left(\tilde{\Omega}_{max}^2\frac{a_1^1P_B||S_C||\cdot||\Gamma_C||\cdot||C_{\bar{\xi}}||}{||K_{\xi,\Omega}||+||K_{u_p}||}\right)||\chi(t)||^3\\
&&&-\left(q_0-3\tilde{\Omega}_{max}P_B||C_{\bar{\xi}}||\right)||\chi(t)||^2\\
&&&+\left(\left(2P_Z+\tilde{\Omega}_{max}\frac{P_B||B_{\xi,\Omega}||\cdot||C_{\bar{\xi}}||}{||K_{\xi,\Omega}||+||K_{u_p}||}\right)\right.\\
&&&\left.\times z_{cmd,max}\right)||\chi(t)||
\end{alignedat}
\end{equation*}
Using the definition of $\tilde{\Omega}_{max}$ and condition (\textit{ii}) of Theorem \ref{t:One}, implies the following:
\begin{align*}
\begin{split}
&\dot{W}(\chi(t))<0,\quad\text{for}\\ &\frac{\kappa_5-\sqrt{\kappa_5^2-4\kappa_9\kappa_{10}}}{2\kappa_9}<||\chi(t)||<\frac{\kappa_5+\sqrt{\kappa_5^2-4\kappa_9\kappa_{10}}}{2\kappa_9}
\end{split}
\end{align*}
Choosing $\chi_{min},\chi_{max}$ from equations (\ref{e:chi_min}) and (\ref{e:chi_max}) implies the following:
\begin{equation}
\label{e:32}
\dot{W}(\chi(t))<0,\quad \forall\chi(t)\in \mathcal{A}\text{ in case 3, sub-case (\textit{ii})}
\end{equation}

\subsubsection{Sub-case iii}
\begin{equation*}
2\chi^T(t)P_{cl}C_{\Delta \bar{u}_2}\frac{1}{a_1^1}\mathbb{U}_r\tau u_r(t)>0
\end{equation*}
The condition of this sub-case implies the following inequality:
\begin{equation*}
2\chi^T(t)P_{cl}C_{\Delta \bar{u}_2}\frac{1}{a_1^1}\mathbb{U}_r\tau u_r(t)<2\chi^T(t)P_{cl}C_{\Delta \bar{u}_2}\frac{1}{a_1^1}\tau u_r(t)
\end{equation*}
Add terms to create $\dot{W}(\chi(t))$ on the left hand side:
\begin{equation*}
\begin{alignedat}{2}
\dot{W}(\chi(t))&\leq&~&-\chi^T(t)Q_{cl}\chi(t)-2\chi^T(t)P_{cl}B_{\Omega}\tilde{\Omega}^{T}(t)C_{\bar{\xi}}\chi(t)\\
&&&+2\chi^T(t)P_{cl}B_Zz_{cmd}(t)
\end{alignedat}
\end{equation*}
The right hand side may be maximized as:
\begin{equation*}
\begin{alignedat}{2}
\dot{W}(\chi(t))&\leq&~&-\left(q_0-2\tilde{\Omega}_{max}P_B||C_{\bar{\xi}}||\right)||\chi(t)||^2\\
&&&+\left(2P_Zz_{cmd,max}\right)||\chi(t)||
\end{alignedat}
\end{equation*}
Using the definition of $\tilde{\Omega}_{max}$ and condition (\textit{ii}) of Theorem \ref{t:One}, implies the following:
\begin{equation*}
\dot{W}(\chi(t))<0,\quad ||\chi(t)||>\frac{\kappa_{11}}{\kappa_{12}}
\end{equation*}
Choosing $\chi_{min}$ from equation (\ref{e:chi_min}) implies the following:
\begin{equation}
\label{e:33}
\dot{W}(\chi(t))<0,\quad \forall\chi(t)\in \mathcal{A}\text{ in case 3, sub-case (\textit{iii})}
\end{equation}
From equations (\ref{e:1}), (\ref{e:21}), (\ref{e:22}), (\ref{e:23}), (\ref{e:31}), (\ref{e:32}), (\ref{e:33}) it can be concluded that:
\begin{equation}
\dot{W}(\chi(t))<0,\quad \forall\chi(t)\in \mathcal{A}
\end{equation}


\section*{Acknowledgment}

This work was supported by the Air Force Research Laboratory, Collaborative Research and Development for Innovative Aerospace Leadership (CRDInAL), Thrust 3 - Control Automation and Mechanization grant FA 8650-16-C-2642 and the Boeing Strategic University Initiative. The authors acknowledge Dr. Daniel Wiese for providing useful hypersonic vehicle numerical simulation scripts.

\ifCLASSOPTIONcaptionsoff
  \newpage
\fi



\bibliographystyle{IEEEtran}
\bibliography{IEEEabrv,References}

\begin{thebibliography}{10}
\providecommand{\url}[1]{#1}
\csname url@samestyle\endcsname
\providecommand{\newblock}{\relax}
\providecommand{\bibinfo}[2]{#2}
\providecommand{\BIBentrySTDinterwordspacing}{\spaceskip=0pt\relax}
\providecommand{\BIBentryALTinterwordstretchfactor}{4}
\providecommand{\BIBentryALTinterwordspacing}{\spaceskip=\fontdimen2\font plus
\BIBentryALTinterwordstretchfactor\fontdimen3\font minus
  \fontdimen4\font\relax}
\providecommand{\BIBforeignlanguage}[2]{{%
\expandafter\ifx\csname l@#1\endcsname\relax
\typeout{** WARNING: IEEEtran.bst: No hyphenation pattern has been}%
\typeout{** loaded for the language `#1'. Using the pattern for}%
\typeout{** the default language instead.}%
\else
\language=\csname l@#1\endcsname
\fi
#2}}
\providecommand{\BIBdecl}{\relax}
\BIBdecl

\bibitem{Sofrony_2006}
J.~Sofrony, M.~C. Turner, I.~Postlethwaite, O.~Brieger, and D.~Leibling,
  ``Anti-windup synthesis for {PIO} avoidance in an experimental aircraft,'' in
  \emph{Proceedings of the 45th {IEEE} Conference on Decision and
  Control}.\hskip 1em plus 0.5em minus 0.4em\relax {IEEE}, 2006.

\bibitem{Duda_1995}
H.~Duda, ``Effects of rate limiting elements in flight control systems - a new
  {PIO}-criterion,'' in \emph{{AIAA} Guidance, Navigation, and Control
  Conference}.\hskip 1em plus 0.5em minus 0.4em\relax AIAA Paper 1995-3204, aug
  1995.

\bibitem{Amato}
F.~Amato, R.~Iervolino, M.~Pandit, S.~Scala, and L.~Verde, ``Analysis of
  pilot-in-the-loop oscillations due to position and rate saturations,'' in
  \emph{Proceedings of the 39th {IEEE} Conference on Decision and
  Control}.\hskip 1em plus 0.5em minus 0.4em\relax {IEEE}, 2000.

\bibitem{Rundqwist_1996}
L.~Rundqwist and K.~Stahl-Gunnarsson, ``Phase compensation of rate limiters in
  unstable aircraft,'' in \emph{Proceeding of the 1996 {IEEE} International
  Conference on Control Applications}.\hskip 1em plus 0.5em minus 0.4em\relax
  {IEEE}, 1996.

\bibitem{Rundqwist_1996a}
L.~Rundqwist and R.~Hillgren, ``Phase compensation of rate limiters in {JAS} 39
  gripen,'' in \emph{21st Atmospheric Flight Mechanics Conference}.\hskip 1em
  plus 0.5em minus 0.4em\relax American Institute of Aeronautics and
  Astronautics, jul 1996.

\bibitem{Narendra_1989}
K.~S. Narendra and A.~M. Annaswamy, \emph{Stable Adaptive Systems}.\hskip 1em
  plus 0.5em minus 0.4em\relax NJ: Prentice-Hall, Inc., 1989, (out of print).

\bibitem{Sastry_1989}
S.~Sastry and M.~Bodson, \emph{Adaptive Control: Stability, Convergence and
  Robustness}.\hskip 1em plus 0.5em minus 0.4em\relax Prentice-Hall, 1989.

\bibitem{Astrom_1995}
K.~J. {\AA}str{\"o}m and B.~Wittenmark, \emph{Adaptive Control: Second
  Edition}.\hskip 1em plus 0.5em minus 0.4em\relax Addison-Wesley Publishing
  Company, 1995.

\bibitem{Ioannou1996}
P.~A. Ioannou and J.~Sun, \emph{Robust Adaptive Control}.\hskip 1em plus 0.5em
  minus 0.4em\relax PTR Prentice-Hall, 1996.

\bibitem{Narendra2005}
K.~S. Narendra and A.~M. Annaswamy, \emph{Stable Adaptive Systems}.\hskip 1em
  plus 0.5em minus 0.4em\relax Prentice Hall, 2005.

\bibitem{karason1994}
S.~P. Karason and A.~M. Annaswamy, ``Adaptive control in the presence of input
  constraints,'' \emph{{IEEE} Transactions on Automatic Control}, vol.~39,
  no.~11, pp. 2325--2330, 1994.

\bibitem{Annaswamy_1995}
A.~M. Annaswamy and S.~P. Karason, ``Discrete-time adaptive control in the
  presence of input constraints,'' \emph{Automatica}, vol.~31, no.~10, pp.
  1421--1431, oct 1995.

\bibitem{Schwager2005}
M.~Schwager, ``Towards verifiable adaptive control for safety critical
  applications,'' Master's thesis, MIT, 2005.

\bibitem{Jang2009}
J.~Jang, ``Adaptive control design with guaranteed margins for nonlinear
  plants,'' Ph.D. dissertation, MIT, 2009.

\bibitem{Lavretsky2007}
E.~Lavretsky and N.~Hovakimyan, ``Stable adaptation in the presence of actuator
  constraints with flight control applications,'' \emph{Journal of Guidance,
  Control, and Dynamics}, vol.~30, no.~2, pp. 337--345, mar 2007.

\bibitem{Serrani2009}
A.~Serrani, A.~M. Zinnecker, L.~Fiorentini, M.~A. Bolender, and D.~B. Doman,
  ``Integrated adaptive guidance and control of constrained nonlinear
  air-breathing hypersonic vehicle models,'' in \emph{2009 American Control
  Conference}.\hskip 1em plus 0.5em minus 0.4em\relax {IEEE}, 2009.

\bibitem{Hess1997}
R.~A. Hess and S.~A. Snell, ``Flight control system design with rate saturating
  actuators,'' \emph{Journal of Guidance, Control, and Dynamics}, vol.~20,
  no.~1, pp. 90--96, jan 1997.

\bibitem{Barbu_2005}
C.~Barbu, S.~Galeani, A.~R. Teel, and L.~Zaccarian, ``Non-linear anti-windup
  for manual flight control,'' \emph{International Journal of Control},
  vol.~78, no.~14, pp. 1111--1129, sep 2005.

\bibitem{Galeani2008}
S.~Galeani, S.~Onori, A.~R. Teel, and L.~Zaccarian, ``A magnitude and rate
  saturation model and its use in the solution of a static anti-windup
  problem,'' \emph{Systems {\&} Control Letters}, vol.~57, no.~1, pp. 1--9, jan
  2008.

\bibitem{Kahveci2008}
N.~E. Kahveci and P.~A. Ioannou, ``Indirect adaptive control for systems with
  input rate saturation,'' in \emph{2008 American Control Conference}.\hskip
  1em plus 0.5em minus 0.4em\relax {IEEE}, jun 2008.

\bibitem{Matsutani2010}
M.~Matsutani, A.~Annaswamy, and L.~G. Crespo, ``Adaptive control in the
  presence of rate saturation with application to a transport aircraft model,''
  in \emph{{AIAA} Guidance, Navigation, and Control Conference}.\hskip 1em plus
  0.5em minus 0.4em\relax American Institute of Aeronautics and Astronautics,
  aug 2010.

\bibitem{Astrom_1989}
K.~J. Astrom and L.~Rundqwist, ``Integrator windup and how to avoid it,'' in
  \emph{1989 American Control Conference}.\hskip 1em plus 0.5em minus
  0.4em\relax {IEEE}, jun 1989.

\bibitem{Astrom_2008}
K.~J. Astrom and B.~Wittenmark, \emph{Adaptive Control: Second Edition}.\hskip
  1em plus 0.5em minus 0.4em\relax Dover, 2008.

\bibitem{Gaudio_2018}
J.~E. Gaudio, A.~M. Annaswamy, and E.~Lavretsky, ``Adaptive control of
  hypersonic vehicles in the presence of rate limits,'' in \emph{2018 {AIAA}
  Guidance, Navigation, and Control Conference}.\hskip 1em plus 0.5em minus
  0.4em\relax American Institute of Aeronautics and Astronautics, jan 2018.

\bibitem{Zheng2013}
Z.~Qu, A.~Annaswamy, and E.~Lavretsky, ``An adaptive controller for very
  flexible aircraft,'' in \emph{{AIAA} Guidance, Navigation, and Control
  ({GNC}) Conference}.\hskip 1em plus 0.5em minus 0.4em\relax American
  Institute of Aeronautics and Astronautics, aug 2013.

\bibitem{Qu2015}
Z.~Qu, A.~M. Annaswamy, and E.~Lavretsky, ``Adaptive output-feedback control
  for relative degree two systems based on closed-loop reference models,'' in
  \emph{2015 54th {IEEE} Conference on Decision and Control ({CDC})}.\hskip 1em
  plus 0.5em minus 0.4em\relax {IEEE}, dec 2015.

\bibitem{Qu2016}
------, ``Adaptive output-feedback control for a class of
  multi-input-multi-output plants with applications to very flexible
  aircraft,'' in \emph{2016 American Control Conference ({ACC})}.\hskip 1em
  plus 0.5em minus 0.4em\relax {IEEE}, jul 2016.

\bibitem{Qu2016a}
Z.~Qu, ``Adaptive output-feedback control and applications to very flexible
  aircraft,'' Ph.D. dissertation, MIT, 2016.

\bibitem{Lavretsky2013}
E.~Lavretsky and K.~A. Wise, \emph{Robust and Adaptive Control with Aerospace
  Applications}.\hskip 1em plus 0.5em minus 0.4em\relax Springer, 2013.

\bibitem{Lavretsky2012}
E.~Lavretsky, ``Adaptive output feedback design using asymptotic properties of
  {LQG}/{LTR} controllers,'' \emph{{IEEE} Transactions on Automatic Control},
  vol.~57, no.~6, pp. 1587--1591, jun 2012.

\bibitem{Wise2013}
K.~A. Wise and E.~Lavretsky, ``Flight control design using observer-based loop
  transfer recovery,'' in \emph{{AIAA} Guidance, Navigation, and Control
  ({GNC}) Conference}.\hskip 1em plus 0.5em minus 0.4em\relax American
  Institute of Aeronautics and Astronautics, aug 2013.

\bibitem{Wise_2018}
K.~A. Wise, ``Design parameter tuning in adaptive observer-based flight control
  architectures,'' in \emph{2018 {AIAA} Information Systems-{AIAA} Infotech @
  Aerospace}.\hskip 1em plus 0.5em minus 0.4em\relax American Institute of
  Aeronautics and Astronautics, jan 2018.

\bibitem{Lavretsky_2019}
E.~Lavretsky, ``Design, analysis, and flight evaluation of a primary control
  system with observer-based loop transfer recovery and direct adaptive
  augmentation for the calspan variable stability simulator learjet-25b
  aircraft,'' in \emph{{AIAA} Scitech 2019 Forum}.\hskip 1em plus 0.5em minus
  0.4em\relax American Institute of Aeronautics and Astronautics, jan 2019.

\bibitem{Gibson2012}
T.~E. Gibson, A.~M. Annaswamy, and E.~Lavretsky, ``Improved transient response
  in adaptive control using projection algorithms and closed loop reference
  models,'' in \emph{{AIAA} Guidance, Navigation, and Control
  Conference}.\hskip 1em plus 0.5em minus 0.4em\relax AIAA Paper 2012-4775, aug
  2012.

\bibitem{Gibson2013a}
------, ``Adaptive systems with closed-loop reference models, part i: Transient
  performance,'' in \emph{American Control Conference}.\hskip 1em plus 0.5em
  minus 0.4em\relax {IEEE}, jun 2013, pp. 3376--3383.

\bibitem{Gibson2013}
------, ``On adaptive control with closed-loop reference models: Transients,
  oscillations, and peaking,'' \emph{{IEEE} Access}, vol.~1, pp. 703--717,
  2013.

\bibitem{Gibson2014}
T.~E. Gibson, ``Closed-loop reference model adaptive control: with application
  to very flexible aircraft,'' Ph.D. dissertation, MIT, 2014.

\bibitem{Stevens2003}
B.~L. Stevens and F.~L. Lewis, \emph{Aircraft Control and Simulation}.\hskip
  1em plus 0.5em minus 0.4em\relax Wiley, 2003.

\bibitem{Russell_2003}
\BIBentryALTinterwordspacing
R.~S. Russell, ``Non-linear f-16 simulation using simulink and matlab,''
  University of Minnesota, Tech. Rep., 2003. [Online]. Available:
  \url{http://www.aem.umn.edu/~balas/darpa_sec/software/F16Manual.pdf}
\BIBentrySTDinterwordspacing

\bibitem{Nguyen_1979}
L.~T. Nguyen, M.~E. Ogburn, W.~P. Gilbert, K.~S. Kibler, P.~W. Brown, and P.~L.
  Deal, ``Simulator study of stall/post-stall characteristics of a fighter
  airplane with relaxed longitudinal static stability,'' \emph{{NASA} Technical
  Paper 1538}, 1979.

\bibitem{Wiese_2015}
D.~P. Wiese, A.~M. Annaswamy, J.~A. Muse, M.~A. Bolender, and E.~Lavretsky,
  ``Adaptive output feedback based on closed-loop reference models for
  hypersonic vehicles,'' \emph{Journal of Guidance, Control, and Dynamics},
  vol.~38, no.~12, pp. 2429--2440, dec 2015.

\bibitem{Wiese_2016}
D.~P. Wiese, ``Systematic adaptive control design using sequential loop
  closure,'' Ph.D. dissertation, Massachusetts Institute of Technology, 2016.

\bibitem{Wiese_2017}
D.~P. Wiese, A.~M. Annaswamy, J.~A. Muse, M.~A. Bolender, and E.~Lavretsky,
  ``Sequential loop closure based adaptive output feedback,'' \emph{{IEEE}
  Access}, vol.~5, pp. 23\,436--23\,451, 2017.

\bibitem{Slotine1991}
J.-J.~E. Slotine and W.~Li, \emph{Applied Nonlinear Control}.\hskip 1em plus
  0.5em minus 0.4em\relax Prentice-Hall, 1991.

\bibitem{Hussain_2017}
H.~S. Hussain, ``Robust adaptive control in the presence of unmodeled
  dynamics,'' Ph.D. dissertation, MIT, 2017.

\end{thebibliography}
\end{document}